\documentclass[a4paper]{article}

\usepackage{xargs}                      
\usepackage[pdftex,dvipsnames]{xcolor}  
\usepackage[colorinlistoftodos,prependcaption,textsize=tiny]{todonotes}
\newcommandx{\unsure}[2][1=]{\todo[linecolor=red,backgroundcolor=red!25,bordercolor=red,#1]{#2}}
\newcommandx{\refrequest}[2][1=]{\todo[linecolor=blue,backgroundcolor=blue!25,bordercolor=blue,#1]{#2}}
\newcommandx{\note}[2][1=]{\todo[linecolor=OliveGreen,backgroundcolor=OliveGreen!25,bordercolor=OliveGreen,#1]{#2}}
\newcommandx{\improvement}[2][1=]{\todo[linecolor=Plum,backgroundcolor=Plum!25,bordercolor=Plum,#1]{#2}}
\newcommandx{\thiswillnotshow}[2][1=]{\todo[disable,#1]{#2}}

\usepackage{tikz-cd}
\usetikzlibrary{cd}
\usepackage[arrow, matrix, curve]{xy}

\usepackage{amssymb}
\usepackage{amsmath}
\usepackage{amsthm} 
\usepackage{mathtools}

\usepackage{hyperref}
\hypersetup{
    colorlinks=true,
    linkcolor=blue,
    filecolor=magenta,      
    urlcolor=cyan,
}

\usepackage[utf8]{inputenc}
\usepackage{hyperref}
\hypersetup{
    colorlinks=true,
    linkcolor=blue,
    filecolor=magenta,      
    urlcolor=cyan,
}
\usepackage{amsmath, amssymb, amsthm, bbm, graphicx, cleveref, caption, tikz, enumerate, todonotes,comment}
\usepackage{subcaption}
\usepackage{wrapfig}

\usepackage{geometry}
\usepackage{mathtools}
\usepackage{pdfpages}

\DeclarePairedDelimiter\abs{\lvert}{\rvert}

\DeclareMathOperator{\Id}{Id}

\DeclareMathOperator{\diver}{div}

\newtheorem{theorem}{Theorem}
\newtheorem{lemma}{Lemma}
\newtheorem{prop}{Proposition}

\newtheorem{definition}{Definition}
\newtheorem*{remark}{Remark}

\renewcommand*{\Re}{\operatorname{Re}}
\renewcommand*{\Im}{\operatorname{Im}}

\renewcommand{\epsilon}{\varepsilon}

\title{Design of defected non-hermitian chains of resonator dimers for spatial and spatio-temporal localizations\thanks{\footnotesize
This work was supported in part by the Swiss National Science Foundation grant number
200021--200307.}}

\author{Habib Ammari\thanks{\footnotesize Department of Mathematics, ETH Z\"urich, R\"amistrasse 101, CH-8092 Z\"urich, Switzerland (habib.ammari@math.ethz.ch, thea.kosche@sam.math.ethz.ch).} \and Erik Orvehed Hiltunen\thanks{\footnotesize Department of Mathematics, Yale University, New Haven, USA (erik.hiltunen@yale.edu).} \and Thea Kosche\footnotemark[2]}

\date{}

\begin{document}
	\maketitle

\begin{abstract}
The aim of this article is to advance the field of metamaterials by proposing formulas for the design of high-contrast metamaterials with prescribed subwavelength defect mode eigenfrequencies. This is achieved in two settings: 
(i) design of non-hermitian static materials and (ii) design of instantly changing non-hermitian time-dependent materials.
The design of static materials is achieved via characterizing equations for the defect mode eigenfrequencies in the setting of a defected dimer material. These characterizing equations are the basis for obtaining formulas for the material parameters of the defect which admit given defect mode eigenfrequencies. Explicit formulas are provided in the setting of one and two given defect mode eigenfrequencies in the setting of a defected chain of dimers.
In the time-dependent case, we first analyze the influence of time-boundaries on the subwavelength solutions. We find that subwavelength solutions are preserved if and only if the material parameters satisfy a temporal Snell's law across the time boundary. The same result  also identifies the change of the time-frequencies uniquely.
Combining this result with those on the design of static 
materials, we obtain an explicit formula for the material design of instantly changing defected dimer materials which admit subwavelength modes with prescribed time-dependent defect mode eigenfrequency.
Finally, we use this formula to create materials which admit spatio-temporally localized defect modes.
   
\end{abstract}

\noindent{\textbf{Mathematics Subject Classification (MSC2000):} 35J05, 35C20, 35P20, 74J20}
		
\vspace{0.2cm}
		
\noindent{\textbf{Keywords:}} defect mode, defect mode eigenfrequency, non-hermitian metamaterial,  spatio-temporal localization, 
metamaterial design
	\vspace{0.5cm}
		
\tableofcontents

\section{Introduction}

In the classical setting, wave localization in space is typically constrained by the diffraction limit. Defects on length-scales smaller than the wavelength have a negligible effect on wave scattering. However, recent developments of \emph{subwavelength} wave localization are able to overcome this barrier \cite{ammari2021functional,sheinfux2017observation, herzig2016interplay,yves2017crytalline,sheng,experiment2013,alu_rev1} by making use of {locally resonant} periodic materials.
Such systems are periodic systems of \emph{high-contrast} resonators which exhibit {subwavelength} resonant frequencies \cite{ammari2018minnaert, phononic1}. They are typical examples of \emph{metamaterials}: materials with a repeating micro-structure that exhibit properties surpassing those of the individual building blocks \cite{milton}. 

As first shown in \cite{bandgap}, \emph{subwavelength band gaps} can be found in any type of these periodic systems. Introducing defects in a subwavelength band gap material can be used to trap or guide waves at subwavelength scales. 
Throughout this paper, subwavelength resonant frequencies of a resonator structure with a defect, which lie inside a subwavelength band gap of the unperturbed structure, will be called subwavelength \emph{defect mode eigenfrequencies}; see  Definition \ref{def:defect_mode_eigenfrequency}. 

In \cite{point_defect}, the existence of subwavelength defect mode eigenfrequencies in point defect structures is demonstrated. In \cite{line_defect}, by showing that the defect modes in the case of a line defect in a crystal of subwavelength resonators are not bound along the defect line, it is proved that  
the line defect indeed acts as a waveguide; waves of certain frequencies inside the band gap  are localized to, and guided along, the line defect. However, there is a fundamental restriction of the practical applicability of these localized/guided modes since the band-gap frequencies are  exponentially close to the edge of the bulk bands.
Therefore, for structures to have physically achievable localization properties, their guiding properties must be robust against imperfections of their design. 

By taking  inspiration from the field of topological insulators, it is shown in \cite{ammari2020topological,ammari2020robust} that systems of subwavelength resonators can be designed to have spatial wave localization properties, at subwavelength scales, that are robust with respect to random imperfections. Moreover, as recently shown in \cite{anderson}, the scattering of time-harmonic waves by systems of high-contrast resonators with randomly chosen material parameters reproduces the characteristic features of Anderson localization and its properties can be understood by considering the phenomena of {hybridization} and {level repulsion}.

On the other hand, it has been proposed in recent years to add non-hermiticity into systems of subwavelength resonators in order to achieve intriguing phenomena, such as nonreciprocal transmission properties, unidirectional amplification
\cite{alu_rev2,mathias_1,mathias_2,shalaev,fleury,jc1,jc2,jc3,sima,NHtopology,fort}. This has been done by either adding gain and loss in the material parameters \cite{sima,mathias_2,NHtopology} or by varying them in time. Time variations of the material parameters could be either instantaneous as considered in this paper (see also \cite{koutserimpas2018electromagnetic, fort,mathias_1}) or periodically modulated \cite{ammari2020time,refedge1,alu_rev2,floq2,thea2022jde}. 

The focus of this article is to furnish explicit and easy-to-implement formulas which allow the design and construction of defected materials with prescribed defect mode eigenfrequencies. This is achieved for up to two defect mode eigenfrequencies in a defected chain of dimers. The method for a more general setting (arbitrarily finitely many defect mode eigenfrequencies in $n$-resonator periodic systems) is also explained.
The design of (static) defected non-hermitian materials that admit a certain defect mode eigenfrequency provides the basis for the design of time-dependent materials that admit time-dependent defect mode eigenfrequencies $\omega(t) = \omega^{-}\chi_{t<0}+\omega^{+}\chi_{t\geq0}$, thus providing an approach to design time-dependent non-hermitian materials.
An application of the design of time-dependent non-hermitian materials is the creation of simultaneous spatial and temporal localization of scalar waves at given (time-dependent) defect mode eigenfrequencies $\omega(t)$. In fact, this is achieved by choosing $\omega(t) = \omega^{-}\chi_{t<0}+\omega^{+}\chi_{t\geq0}$ with $\Im(\omega^{-}) > 0$ and $\Im(\omega^{+}) < 0$.

For the sake of clarity and presentation, we will consider a chain of periodically repeated resonator dimers and introduce one to two defected resonators to the periodic structure. However, we would like to emphasize that the choice of a chain with repeated dimers has \emph{only} illustrative reasons and that the same approach presented here can be used to consider a system of $n$ resonators which are repeated with some arbitrary lattice $\Lambda \subset \mathbb{R}^3$; see the two remarks at the end of Subsection \ref{sec:Relation_between_defect_properties_and_localized_frequency}. 

Thus, for the sake of clarity, throughout this paper, we  let $D_1$, $D_2 \subset [0,1) \times \mathbb{R}^2$ be two disjoint smooth simply connected domains and consider $D_i^m := D_i + (m,0,0)^T$ with $i= 1,2$ and $\mathcal{D} := \bigcup_{m \in \mathbb{Z}} D_1^m \cup D_2^m$.
We will study subwavelength waves in the setting where the material parameters are of the form
    \begin{align}
        \rho(x,t) = \begin{cases}
            \rho_b^- & t<0,~ x \in \mathbb{R}^3\setminus\mathcal{D} \\
            \rho_{i,m}^{-} & t<0,~ x \in D_i^m \\
            \rho_b^+ & t\geq 0,~ x \in \mathbb{R}^3\setminus\mathcal{D} \\
            \rho_{i,m}^{+} & t\geq 0,~ x \in D_i^m \\
        \end{cases} \quad \text{ and } \quad
        \kappa(x,t) = \begin{cases}
            \kappa_b^- & t<0,~ x \in \mathbb{R}^3\setminus\mathcal{D} \\
            \kappa_{i,m}^{-} & t<0,~ x \in D_i^m \\
            \kappa_b^+ & t\geq 0,~ x \in \mathbb{R}^3\setminus\mathcal{D} \\
            \kappa_{i,m}^{+} & t\geq 0,~ x \in D_i^m \\
        \end{cases},
    \end{align}
with $\rho_b^-,~ \rho_{i,m}^{-},~ \rho_b^+,~ \rho_{i,m}^{+},~ \kappa_b^-,~ \kappa_{i,m}^{-},~ \kappa_b^+,~ \kappa_{i,m}^{+} \in \mathbb{C}$.
The associated wave equation with time-dependent coefficients is given by
\begin{equation} \label{eqwtd}
    \left(\frac{\partial}{\partial t}\frac{1}{\kappa(x,t)}\frac{\partial}{\partial t} - \diver_x \frac{1}{\rho(x,t)}\nabla_x\right)u(x,t) = 0, \quad x \in \mathbb{R}^3, \quad t \in \mathbb{R}.
\end{equation}
In this setting, we aim to find solutions which are time-harmonic for negative and positive times. That is, we want to find solutions $u: \mathbb{R}^3 \times \mathbb{R} \rightarrow \mathbb{C}$ of the following form:
\begin{align}\label{eq:sol_intro}
    u(x,t) = \begin{cases}
                v(x)e^{-i\omega_-t} & t<0,\\
                v(x)e^{-i\omega_+t} & t\geq0, 
            \end{cases}
\end{align}
where $v: \mathbb{R}^3 \rightarrow \mathbb{C}$ and $\omega_\pm \in \mathbb{C}$. Such solutions will be called \emph{quasi-harmonic}.
We will say that a solution $u$ is then \emph{spatio-temporally localized}, if $v(\cdot,x_2,x_3) \in L^2(\mathbb{R})$ for almost every $(x_2,x_3) \in \mathbb{R}$ (\emph{spatial localization}) and if for almost every choice of $x \in \mathbb{R}^3$ the function $u(x,\cdot)$ is in $L^2(\mathbb{R})$ (\emph{temporal localization}). This is the case precisely when $\Im(\omega_-) > 0$ and $\Im(\omega_+) < 0$.
In this setting it is possible to split the problem into two parts: a static part and a time-transition part. We will first study $v: \mathbb{R}^3 \rightarrow \mathbb{C}$ more closely and determine a setting in which $v$ is spatially localized. In fact, as shown in Section \ref{sec:Spatio-temporal_localization}, $v$ is the solution to the Helmholtz equation
    \begin{align*}
        \Delta v(x) + \omega_{\pm}\sqrt{\frac{\kappa^{\pm}(x)}{\rho^{\pm}(x)}}v(x) = 0,
    \end{align*}
    together with an appropriate boundary conditions for $v(x)$ on $\partial \mathcal{D}$, 
where the constant (w.r.t. time $t$) coefficients are given by $\kappa^{-}(x) = \kappa(x,t)|_{t<0},~ \rho^{-}(x) = \rho(x,t)|_{t<0}$ and $\kappa^{+}(x) = \kappa(x,t)|_{t\geq0},~ \rho^{+}(x) = \rho(x,t)|_{t\geq0}$,  respectively.

The methodology adopted in this paper is based on a discrete approximation of the spectral problem obtained in the high contrast limit \cite{ammari2021functional}. In a periodic material, e.g. assuming that the material parameters $\rho$ and $\kappa$ are periodic with respect to the lattice $\Lambda = (1,0,0)^T\mathbb{Z}$, that is, $\rho_{i,m} = \rho_{i,n}$ and $\kappa_{i,m} = \kappa_{i,n}$ for all $m,n \in \mathbb{Z}$, and $\rho_{i,m}/\rho_b \ll 1$, it has been shown in 
\cite{ammari2021functional,ammari2020topological} that the subwavelength spectrum can be approximated in terms of the eigenvalues of the so-called \emph{generalized capacitance matrix} $\mathcal{C}^\alpha$, where $\alpha \in \mathbb{R}/2\pi\mathbb{Z}$. 
In a periodic structure the associated (Bloch) modes $v \in L^2_{loc}$ are $\alpha$-quasiperiodic, which in this case means that $(x_1,x_2,x_3) \mapsto v(x_1,x_2,x_3)\exp(-i\alpha \cdot x_1)$ is $\Lambda$-periodic and thus $v(x)$ is not square integrable along the $x_1$-direction. In order to achieve spatial localization, it is thus necessary to break the periodicity of the system and to introduce defects in the periodic structure.
One possibility to break the periodicity is to change some of the $\rho_{m,i}$'s and $\kappa_{m,i}$'s. It turns out that the defect mode eigenfrequencies associated to a defected structure with finitely many defects are related to the generalized capacitance matrix $\mathcal{C}^\alpha$ of the periodic non-defected structure and the defect material parameters and locations \cite{anderson}. 

Based on the characterization of the defect mode eigenfrequencies in the hermitian case (i.e., when all the material parameters are real positive) first provided in \cite{anderson},  we derive several characterizing equations for the defect mode eigenfrequencies of localized modes in the setting of single and double defects in non-hermitian chains of resonator dimers; see Subsection \ref{sec:Chains_of_dimers_with_defect}. Using these characterizing equations we succeed in showing numerically the existence of localized modes in defected non-hermitian chains of dimers. 

Changing perspective, we ask no longer whether a certain choice of resonator system with given material parameters admits localized solutions, but we ask whether it is possible to introduce a defect (by changing the material parameters of a single resonator) in a given resonator structure, such that a given frequency $\omega$ is a defect mode eigenfrequency of the defected system. 
The characterization from Subsection \ref{sec:Chains_of_dimers_with_defect} provides a basis for answering this question. In Subsection \ref{sec:Relation_between_defect_properties_and_localized_frequency}, we prove that the answer to this question is indeed `yes' for frequencies $\omega$ which lie in the band gap of the Helmholtz problem and which satisfy a mild technical condition. We are able to derive an explicit function $\omega \mapsto V^{def}(\omega)$ which maps the frequency to the defect material parameters; see Theorem \ref{thm:formula_defect_freq}. Even more is true. Theorem \ref{thm:formula_double_defect_freq} provides a function $(\omega_1,\omega_2) \mapsto (V_1^{def}(\omega_1,\omega_2),V_2^{def}(\omega_1,\omega_2))$ which maps a pair of frequencies $(\omega_1,\omega_2)$ to a pair of defect material parameters $(V_1^{def}(\omega_1,\omega_2),V_2^{def}(\omega_1,\omega_2))$, such that the material with two defected resonators with material parameters given by $V_1^{def}(\omega_1,\omega_2)$ and $V_2^{def}(\omega_1,\omega_2)$, respectively, admits two defect mode eigenfrequencies given by $\omega_1$ and $\omega_2$.
We observe that the procedure to obtain functions which map frequencies to associated defected material parameters reduces to solving polynomial equations in the variables given by the defect material parameters. Indeed for $n$ given frequencies, it is possible to find $n$ defect material parameters such that, generically, the given frequencies occur as defect mode eigenfrequencies. In that case, the material parameters $V^{def}_1,\ldots,V^{def}_n$ satisfy a system of $n$ polynomial equations of degree $n$ and as such, generically, have a solution. This makes clear that, generically, it is possible to deduce similar explicit formulas as in Theorems \ref{thm:formula_defect_freq} and \ref{thm:formula_double_defect_freq}, for up to four given frequencies. For at least five given frequencies, generically, the system of equations can no longer be solved explicitly, but needs to be solved numerically. For further details see the remarks at the end of Subsection \ref{sec:Relation_between_defect_properties_and_localized_frequency}.

Having established spatial localization, in a second step, we analyse the time-dependent setting in Section \ref{sec:Spatio-temporal_localization}. We show that in order to admit quasi-harmonic solutions as in \eqref{eq:sol_intro} it is sufficient and necessary that $\rho$ and $\kappa$ satisfies a \emph{temporal Snell's law}, that is, $\rho^-\omega^+ = \rho^+\omega^-$ and $\kappa^+\omega^+ = \kappa^-\omega^-$. This gives us a one-to-one correspondence between materials $\mathcal{M}_1$ and $\mathcal{M}_2$ such that the composite time-dependent material $\mathcal{M}$ given by $\mathcal{M}_1$ for times $t < 0$ and by $\mathcal{M}_2$ for times $t \geq 0$, admits solutions of the form \eqref{eq:sol_intro}. The temporal Snell's law describes time-refraction (see, for example, \cite{mathias_1}) tailored so that any time-harmonic solution of $\mathcal{M}_1$ will be refracted to a time-harmonic solution of $\mathcal{M}_2$ across the temporal boundary.

With these results at hand, we succeed to prove Theorem \ref{thm:time_dep_material_design}, which is in the same style as Theorems \ref{thm:formula_defect_freq} and \ref{thm:formula_double_defect_freq}.  Theorem \ref{thm:time_dep_material_design} establishes an explicit and easy-to-implement function which maps a time-dependent frequency $\omega(t) = \omega^{-}\chi_{t<0}+\omega^{+}\chi_{t\geq0}$ as in \eqref{eq:sol_intro} to the material parameters of a defected instantly changing material, such that $\omega(t)$ occurs as a \emph{time-dependent} defect mode eigenfrequency. 
In particular, it allows to construct examples of spatio-temporal localization, by choosing $\Im(\omega^{-})>0$ and $\Im(\omega^{+})<0$, see Theorem \ref{thm:spatio-temporal_loc} in Section \ref{subsec:Spatio-temporal_localization}.

The paper is organized as follows. It is  divided into a first part which treats the case of static metamaterials in Section \ref{sec:spatial_localization} and a second part which builds on the first and treats the case of time-dependent instantly changing metamaterials in Section \ref{sec:Spatio-temporal_localization}.  In the first part, exact formulas characterizing the defect mode eigenfrequencies are derived in Subsections \ref{subsec:PT_symmetric_mat} and \ref{subsec:material_double_and_single_defect} for a  system of resonator dimers with a single or a double defect. Then explicit formulas for the design of defected structures which admit specified defect mode eigenfrequencies are obtained (Subsection \ref{sec:Relation_between_defect_properties_and_localized_frequency}). The results of Subsection \ref{sec:Relation_between_defect_properties_and_localized_frequency} are generalized in Section \ref{sec:Spatio-temporal_localization} to the setting of instantly changing metamaterials with prescribed time-dependent defect mode eigenfrequencies $\omega(t) = \omega^{-}\chi_{t<0}+\omega^{+}\chi_{t\geq0}$. To this end, in Subsection \ref{subsec:Instantly_changing_material}, the occurrence of \emph{quasi-harmonic waves} in instantly changing metamaterials is characterized in Theorem \ref{thm:time_dependent}. Then, Subsection \ref{subsec:temporal_material} introduces the design of time-dependent materials which admit localized modes with a prescribed time-dependent defect mode eigenfrequency $\omega(t) = \omega^{-}\chi_{t<0}+\omega^{+}\chi_{t\geq0}$. The main result is presented in Theorem \ref{thm:time_dep_material_design}, which furnishes the basis for the design of instantly changing metamaterials which admit spatio-temporally localized modes at given frequencies. This is addressed in Section \ref{subsec:Spatio-temporal_localization} and the central design and existence result is presented in Theorem \ref{thm:spatio-temporal_loc}.

\section{Spatial localization}\label{sec:spatial_localization}
    As outlined in the introduction, this section is the first step towards designing instantly changing materials with prescribed time-dependent defect mode eigenfrequency, which is the basis for spatio-temporal localization. After having introduced the setting and necessary background in Subsection \ref{sec:Theoretical_background_and_setting}, we derive in Subsection \ref{sec:Chains_of_dimers_with_defect} characterizing equations for defect mode eigenfrequencies in non-hermitian infinite chains of dimers with one or two defected resonators. These characterizing equations are numerically solved to determine the defect mode eigenfrequencies of  defected structures.
    We also observe that approximating an infinite chain resonator dimers with a finite, truncated chain gives accurate approximations  for the defect mode eigenfrequencies. Then, in Subsection \ref{sec:Relation_between_defect_properties_and_localized_frequency}, we are able to use these characterizing equations for the design of defected dimer systems which admit prescribed defect mode eigenfrequencies. That is, given $\omega$, we derive explicit functions which map frequencies $\omega$ to the defect material parameters of a dimer system, such that  $\omega$ occurs as a defect mode eigenfrequency in the associated defected material.

    \subsection{Theoretical background and setting}\label{sec:Theoretical_background_and_setting}

    \begin{figure}[h]
        \centering
        \begin{subfigure}[b]{0.49\textwidth}
            \centering
            \includegraphics[width=\textwidth]
            {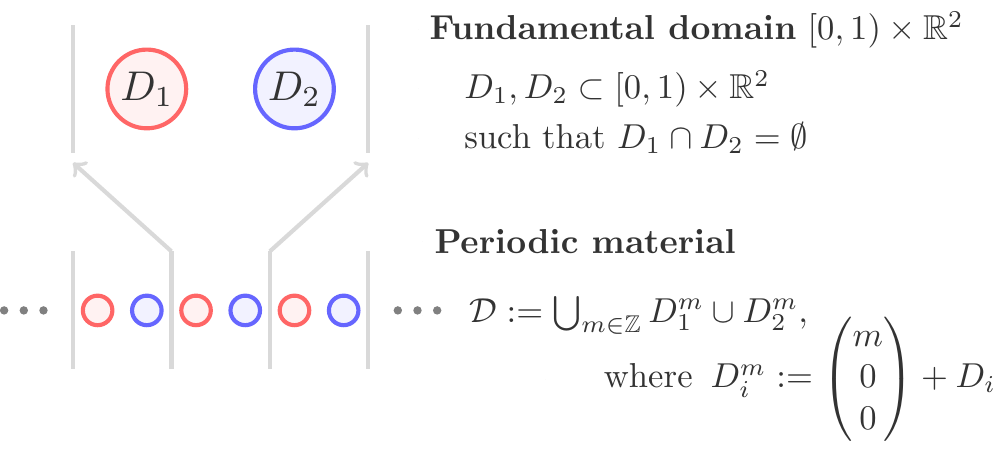}
            \caption*{\vspace{2pt}}
        \end{subfigure}
        \hfill
        \begin{subfigure}[b]{0.49\textwidth}
            \centering
            \includegraphics[width=\textwidth]{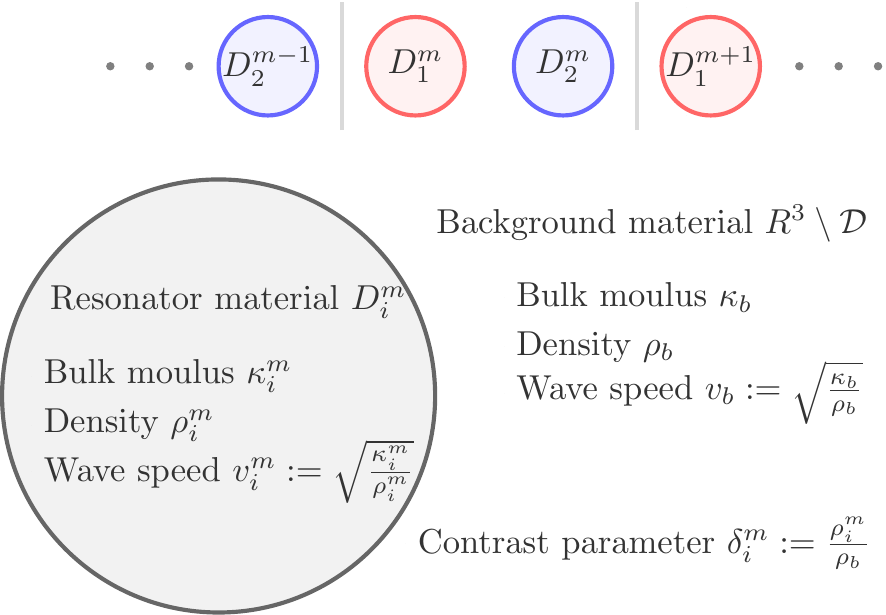}
        \end{subfigure}
        \caption{Notation convention for material parameters used in this paper.}
        \label{fig:def_structure_periodic_material_parameters}
        \label{fig:def_structure_periodic}
    \end{figure}

    In this work, we will consider disjoint resonators $D_1, D_2 \subset [0,1) \times \mathbb{R}^2$ repeated periodically with respect to $\Lambda = (1,0,0)^T\mathbb{Z}$ in $\mathbb{R}^3$, forming a chain of resonator dimers as depicted on Figure \ref{fig:def_structure_periodic} (left). To be more precise, let 
            \begin{align*}
                D_1, D_2 \subset [0,1) \times \mathbb{R}^{2} \quad \text{ such that }\quad D_1 \cap D_2 = \emptyset\\
            \end{align*}    
        and define
            $$\mathcal{D} := \bigcup_{m \in \mathbb{Z}} D^m_1 \cup D^m_2, \quad \text{ where }\quad D_i^m := \begin{pmatrix}
                m \\ 0 \\ 0
            \end{pmatrix} + D_i.$$
    Furthermore, we associate to every resonator $D_i^m$ with $i=1,2$ and $m \in \mathbb{Z}$ the material parameters $\rho_i^m$ and $\kappa_i^m$; $\rho_b$ and $\kappa_b$ will denote the material parameters of the background medium (as depicted in Figure \ref{fig:def_structure_periodic_material_parameters} (right)). The respective \emph{wave speeds} $v_i^m$ in each resonator $D_i^m$ are then given by $v_i^m = \sqrt{\kappa_i^m/\rho_i^m}$ and we will denote by $v_b = \sqrt{\kappa_b/\rho_b}$ the wave speed in the surrounding background medium $\mathbb{R}^3 \setminus \mathcal{D}$. The contrast parameter associated to the resonator $D_i^m$ will be denoted by $\delta_i^m:=\rho_i^m/\rho_b$ and which will satisfy that for some more general $\delta$
            \begin{equation}
                \label{deltall1}
             \delta_i^m = O(\delta) \text{ as } \delta  \rightarrow 0.\end{equation}

    In the static case, we are interested in solving the wave equation
    \begin{align*}
        \left(\frac{1}{\kappa(x)} \frac{\partial^2}{\partial t^2} - \diver_x \frac{1}{\rho(x)}\nabla_x\right)u(x,t) = 0, \quad x \in \mathbb{R}^3, \quad t \in \mathbb{R},
    \end{align*}
    for a time harmonic solution $u(x,t) = v(x)\exp(-i\omega t)$ in the \emph{subwavelength regime}. That is, we are interested in parameterized solutions $u_\delta(x,t) = v_\delta(x)\exp(-i\omega_\delta t)$, such that $v_\delta$ and $\omega_\delta$ depend continuously on $\delta$ and such that $v_\delta \not= 0$ for all $\delta$ close to 0 and
    $$ \omega_\delta \rightarrow 0 \text{ as } \delta \rightarrow 0. $$
    We call solutions $(\omega_\delta, v_\delta)$ of this form \emph{subwavelength solutions}. In that case, $\omega_\delta$ is called \emph{subwavelength frequency} and $v_\delta$ is called \emph{subwavelength mode}.
     Solving the wave equation \eqref{eqwtd} is actually equivalent to solving the following system of Helmholtz equations in $v:\mathbb{R}^3 \rightarrow \mathbb{C}$ and $\omega \in \mathbb{C}$ with $\Re(\omega) > 0$:
        \begin{align}\label{eq:Helmholtz_prob}
            \begin{cases}
                \Delta v + \frac{\omega^2}{v^2} v = 0 & \text{ in } \mathbb{R}^3 \setminus \mathcal{D},\\
                \Delta v + \frac{\omega^2}{(v_i^m)^2} v = 0 & \text{ in } D_i^m,~ m\in \mathbb{Z},~ i =1,2,\\
                v\rvert_+ - v\rvert_- = 0 &\text{ on } \partial \mathcal{D},\\
                \delta_i^m\frac{\partial v}{\partial \nu}\rvert_+ - \frac{\partial v}{\partial \nu}\rvert_- = 0 &\text{ on } \partial \mathcal{D},\\
                v(x_1,x_2, x_3) & \text{ satisfies the outgoing radiation condition as $\lvert x_1 \rvert \rightarrow \infty$.}
            \end{cases}
        \end{align}

    Assuming that the material parameters $\rho_i^m$ and $\kappa_i^m$ are periodic, that is, assuming $\rho_i^m = \rho_i^n$ and $\kappa_i^m = \kappa_i^n$ for $i=1,2$ and for all $m,n \in \mathbb{Z}$, one can apply Floquet-Bloch theory and find quasi-periodic solutions to (\ref{eq:Helmholtz_prob}). 
    In this case, the case of periodic material parameters, the \emph{lattice} of the structure is given by $\Lambda = \mathbb{Z}(1,0,0)^T$, the \emph{dual lattice} is given by $\Lambda^* = \mathbb{Z}(2\pi,0,0)^T$ and the \emph{Brillouin Zone} by $Y^* = \mathbb{R}/2\pi\mathbb{Z} \times \{0\}^{2}$.
    For $\alpha \in Y^*$, define the quasiperiodic Green's function $G^{\alpha}(x)$ as the Floquet transform of the Green's function $G(x)$ associated to the Laplace equation 
        $$ \Delta u = 0, $$
    which is given in three dimensions by 
        $$ G(x) = -\frac{1}{4\pi\abs{x}}.$$
    Hence, $G^{\alpha}$ is defined as
        $$ G^{\alpha}(x) := \sum_{m \in \Lambda} G(x)e^{i \alpha \cdot m}.$$
    The associated single layer potential to $G^{\alpha}$ is then
        $$ \mathcal{S}^\alpha_D[\phi](x) := \int_{\partial D} G^\alpha(x - y)\phi(y)d\sigma(y), \quad x \in \mathbb{R}^3,$$
        and is known to be invertible from $L^2(\partial D)$ onto $H^1(\partial D)$ with $H^1(\partial D)$ being the usual Sobolev space; see, for instance, \cite{ammari2018mathematical}. 

    \begin{definition}[Quasiperiodic generalized capacitance matrix]
        Let $\alpha \in Y^* - \{0\}$. Let $D_1,\ldots,D_N \subset \mathbb{R}^3$ be $N$ resonators which are repeated periodically with respect to some lattice $\Lambda \subset \mathbb{R}^3$. Define the \emph{quasiperiodic capacitance matrix} $C^\alpha = (C_{ij}^\alpha) \in \mathbb{C}^{N\times N}$ as
            $$ C_{ij}^\alpha = - \int_{\partial D_i}(\mathcal{S}^\alpha_D)^{-1}[\chi_{\partial D_j}]d\sigma, \quad i,j = 1,\ldots,N,$$
where $\chi_{\partial D_j}$ is the indicator function of $\partial D_i$. Then the \emph{generalized quasiperiodic capacitance matrix} $\mathcal{C}^\alpha = (\mathcal{C}_{ij}^\alpha) \in \mathbb{C}^{N\times N}$ is given by
            \begin{equation}
            \label{defc} \mathcal{C}_{ij}^\alpha = \frac{\delta_i v_i^2}{\abs{D_i}}C_{ij}^\alpha, \quad i,j = 1,\ldots,N, \end{equation}
            where $\abs{D_i}$ denotes the volume of $D_i$. 
    \end{definition}

    The subwavelength frequencies of the system of $N$ resonators are given by the eigenvalues of the generalized quasiperiodic capacitance matrix up to an error of order $\mathcal{O}(\delta)$; see \cite{ammari2021functional}.
    Thus, {\it a priori}, the quasiperiodic capacitance matrix only describes subwavelength phenomena in the infinite periodic setting. However, it has been proven in \cite{anderson} that the quasiperiodic capacitance matrix also helps to understand localized modes in defected structures with finitely many defected resonators.
    
    More precisely, assume that for finitely many $D_i^m$ the material parameters are not given by $\delta_i v_i^2$ but by $\delta_i v_i^2 (1 + x_i^m)$ for some $x_i^m \in \mathbb{C}$; an example being depicted in Figure \ref{fig:def_structure_Introduction_of_def_details}.
    For $m \in \mathbb{Z}$, denote by $X_m$ the diagonal matrix with the diagonal entries given by 
    $$ (X_m)_{ii} = x_i^m. $$

    \begin{definition}[Localized mode and defect mode eigenfrequency]\label{def:defect_mode_eigenfrequency}
        Using the notation from the beginning of this chapter and being in the setting of a defected material, where the material parameters $\delta_i v_i^2$ of finitely many resonators $D^m_i$ have been replaced by $\delta_i v_i^2 (1 + x_i^m)$ for $x^m_i \in \mathbb{C}$, we consider the Helmholtz problem \eqref{eq:Helmholtz_prob} and assume that $(\omega_\delta, v_\delta)$ is a subwavelength solution.
        We  call $v_\delta$ a \emph{localized} or \emph{defect mode} if $(\omega_\delta,v_\delta)$ satisfies
            \begin{align*}
                \int_{\mathbb{R}} \abs{v_\delta(x_1,x_2,x_3)}^2dx_1 < \infty ~ &\text{ for a.e. } (x_2,x_3) \in \mathbb{R}^2,\\
                \omega_\delta \not \in \bigcup_{\alpha\in Y^*\setminus\{0\}} &\sigma(\mathcal{C}^\alpha),
            \end{align*}
        where $\sigma(\mathcal{C}^\alpha)$ denotes the set of eigenvalues of $\mathcal{C}^\alpha$.
        In that case, $\omega_\delta$ is called \emph{defect mode eigenfrequency}.
    \end{definition}
    The first condition formalizes the localized nature of the subwavelength solution, whereas the second condition formalizes the defect nature. We have the following result from \cite[Proposition 3.7]{anderson}. 
    \begin{figure}
        \centering
        \includegraphics{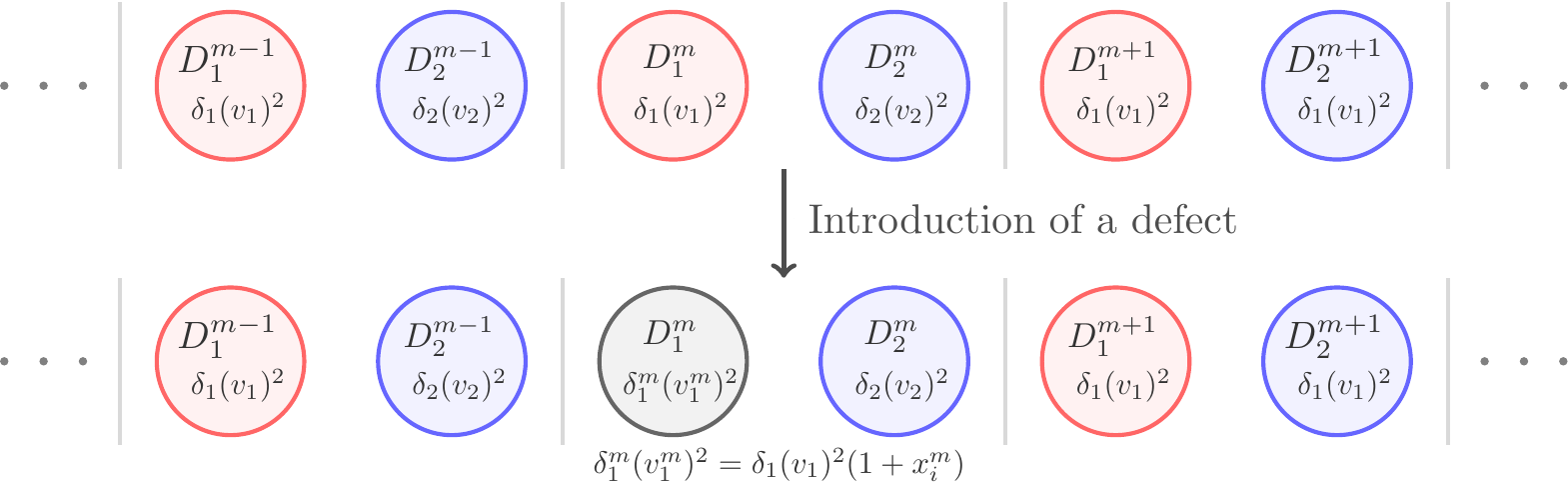}
        \caption{Illustrative example of the introduction of finitely many defected resonators with the notation convention used in Section \ref{sec:Theoretical_background_and_setting}.}
        \label{fig:def_structure_Introduction_of_def_details}
    \end{figure}
    \begin{theorem}[Defect mode eigenfrequencies] \label{thm:defect}
        Assume that $x_i^m \neq 0$ for finitely many $m \in \Lambda$ and $1\leq i \leq N$. Let $\omega \in \mathbb{C}$ and assume $\omega \not\in \bigcup_{\alpha \in Y^*}\sigma(\mathcal{C}^\alpha)$ for all $\alpha \in Y^* - \{0\}$. Then $\omega$ is a defect mode eigenfrequency associated to the above defected material if and only if
        \begin{align}\label{eq:general_eq_for_defect}
            \det(\Id - \mathcal{X}\mathcal{T}(\omega)) = 0,
        \end{align}
        where $\Id$ is the identity matrix and $\mathcal{X}$ is the block-diagonal matrix with entries $X_m$ and $\mathcal{T}$ is the infinite block \emph{Toeplitz matrix} given by
            \begin{align}\label{eq:toeplitz_mat}
                \mathcal{T}(\omega)_{m,n} = T^{n-m}(\omega) = 
                - \frac{1}{\abs{Y^*}}\int_{Y^*} e^{i\alpha (n-m)}\mathcal{C}^\alpha(\mathcal{C}^\alpha - \omega^2 \Id)^{-1}d\alpha.
            \end{align}
    \end{theorem}

    Since $\mathcal{X}$ has only finitely many non-zero values, it follows that the determinant in equation \eqref{eq:general_eq_for_defect} reduces to a determinant of a finite dimensional matrix and is thus well-defined.
    
    Furthermore, we would like to note that this theorem serves as the basis of the following paper. Although \cite{anderson} is mainly concerned with the \emph{hermitian} case, that is, real-valued material parameters, the theorem holds in the general \emph{non-hermitian} case of complex material parameters. This is the setting in which this theorem will be heavily used in this work. Most of the results obtained in this paper rely on its characterization of the defect mode eigenfrequencies.

    \subsection{Characterization of defect mode eigenfrequencies}\label{sec:Chains_of_dimers_with_defect}
        In this section, characterizing equations for defect mode eigenfrequencies in defected structures will be derived. To this end, we will study three different settings, the first being a $\mathcal{PT}$-symmetric chain of spherical resonator dimers 
        (i.e., identical spheres with material parameters complex conjugate to each other as shown in \eqref{complexcj}) and with a single defected resonator $D_1^0$ in which the wave speed $v_1^0$ has been replaced with its complex conjugate $\overline{v_1^0}$. 
        The characterizing equation for the defect mode eigenfrequencies in this setting is the result of Lemma \ref{lem:formula_defect_freq_PT}. Next, we consider the setting of a general chain of dimers with a single defected resonator $D_1^0$ with wave speed $\sqrt{V^{def}/\delta_1^0}$, the respective result can be found in Lemma \ref{lem:formula_defect_freq_single}. 
        Finally, we characterize the defect mode eigenfrequencies in a defected chain of dimers with two defected resonators $D_1^0$ and $D_2^0$; see Lemma \ref{lem:formula_defect_freq_double}. Finally, these results are used in Subsection \ref{sec:Relation_between_defect_properties_and_localized_frequency} in order to design defected materials with prescribed defect mode eigenfrequencies. 

        \subsubsection{PT-symmetric periodic system of resonator 
        dimers with a loss defect}\label{subsec:PT_symmetric_mat}
            In this section, we will study phenomena in the setting of a $\mathcal{PT}$-symmetric chain of equally sized, spherical, disjoint resonators and introduce a loss (respectively, gain) defect by conjugating the material parameter of the resonator $D_1^0$. This setting will significantly simplify the characterizing equation obtained for the defect mode eigenfrequencies.

            \begin{figure}[h]
                \centering
                \includegraphics{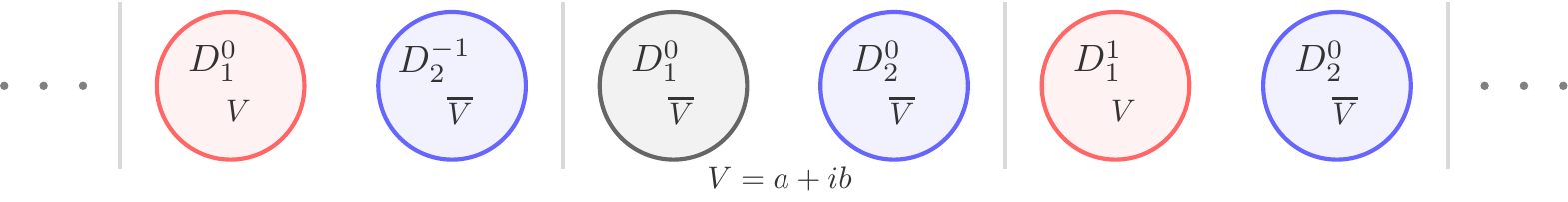}
                \caption{Setting and notation convention used in Lemma \ref{lem:formula_defect_freq_PT}.}
                \label{fig:def_PT_structure_with_defect_simplified}
            \end{figure}
            
            We suppose the material parameters of the periodic material to be given by
            \begin{equation} \label{complexcj}
                \delta_j (v_j)^2 = \begin{cases}
                    a + ib, &\text{if } j = 1,\\
                    a - ib, &\text{if } j = 2.
                \end{cases}
            \end{equation}

            As depicted in Figure \ref{fig:def_PT_structure_with_defect_simplified}, in order to study defected non-hermitian materials, we will replace the material parameter of a single resonator by its complex conjugate thus creating a loss/gain defect as depicted in Figures \ref{fig:def_structure_gain} and \ref{fig:def_structure_loss}.
            
            \begin{figure}[h]
                \centering
                \begin{subfigure}{\textwidth}
                    \centering
                    \includegraphics{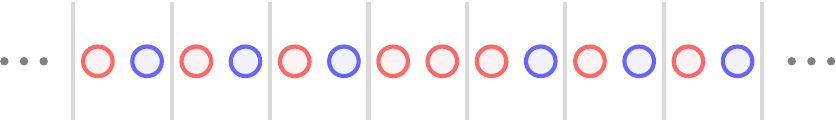}
                    \caption{Non-hermitian periodic material with \emph{gain} defect.}    
                    \label{fig:def_structure_gain}
                \end{subfigure}
                \hfill
                \begin{subfigure}{\textwidth}
                    \centering
                    \includegraphics{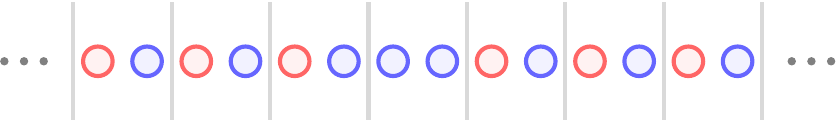}
                    \caption{Non-hermitian periodic material with \emph{loss} defect.}
                    \label{fig:def_structure_loss}
                \end{subfigure}
                \caption{Non-hermitian periodic structure without defect. Red indicates material parameters equal to $a + ib$ whereas blue indicates material parameters equal to $a-ib$. The proportions are precisely those used in the numerical simulations: length of the fundamental domain $L=1$, radius of the spherical resonators $0.15$, positions of the resonators at $(0.25 + \mathbb{Z},0,0)$ and $(0.75 + \mathbb{Z},0,0)$. }
            \end{figure}

            In this case, the generalized quasiperiodic capacitance 
            matrix defined in (\ref{defc})  takes the form 
                \begin{align*}
                    \mathcal{C}^\alpha = \begin{pmatrix}
                        a+ib & 0 \\
                        0 & a-ib
                    \end{pmatrix}C_\alpha = \begin{pmatrix}
                        a+ib & 0 \\
                        0 & a-ib
                    \end{pmatrix}\begin{pmatrix}
                        C_{11}(\alpha) & C_{12}(\alpha) \\
                        \overline{C_{12}(\alpha)} & C_{11}(\alpha)
                    \end{pmatrix}.
                \end{align*}
            In the case of a single loss defect in the zeroth fundamental domain, we obtain that there is precisely one non-zero block $X_0$ in the block diagonal matrix $(X_m)_{m \in \mathbb{Z}} = \mathcal{X}$, which is given by 
                \begin{align*}
                    X_0 = \begin{pmatrix}
                        \vspace{5pt}\frac{-2ib}{a+ib} &0\\
                    0 &0
                    \end{pmatrix}
                \end{align*}
            and the relevant part of the Toeplitz matrix $\mathcal{T}(\omega)$ in \eqref{eq:general_eq_for_defect} reduces to
                \begin{align*}
                    \mathcal{T}(\omega)_{0,0} 
                            ~=~ T^0 
                            ~=~ -\frac{1}{\lvert Y^*\rvert}\int_{Y^*}\mathcal{C}^\alpha(\mathcal{C}^\alpha - \omega^2 \Id)^{-1}d\alpha,
                \end{align*}
            where $\Id$ denotes here the $2\times 2$ identity matrix.

The following lemma holds. 
            \begin{lemma}\label{lem:formula_defect_freq_PT}
                The defect mode eigenfrequencies in the above setting are given by solutions $\omega \in \mathbb{C} \setminus \cup_{\alpha \in Y^*}\sigma(\mathcal{C}^\alpha)$ to the equation
                    \begin{align}\label{eq:omega_def_mode}
                        0 = \int_{Y^*}\frac{
                        \omega^4 - 2(a - ib)(C^\alpha)_{11}\omega^2 + (a-ib)^2\det(C^\alpha)
                        }{
                        \omega^4 - 2a(C^\alpha)_{11}\omega^2 + \lvert(a+ib)\rvert^2 \det(C^\alpha)
                        }d\alpha.
                    \end{align}
            \end{lemma}

            \begin{proof}
                From \Cref{thm:defect}, we know that whenever $\omega \in \mathbb{C} \setminus \cup_{\alpha \in Y^*}\sigma(\mathcal{C}^\alpha)$, then $\omega$ corresponds to a defect mode eigenfrequency if and only if it satisfies the equation
                    \begin{align*}
                        \det(\Id - \mathcal{X}\mathcal{T}(\omega)) = 0.
                    \end{align*}
                Since only one resonator is modified, this reduces to the equation
                    \begin{equation} \label{eq11}
                        1 - X_0^1(\mathcal{T})_{11} = 0.
                    \end{equation}
                Then we compute
                    \begin{align*}
                        &(\mathcal{C}^\alpha(\mathcal{C}^\alpha - \omega^2 \Id)^{-1})_{11}\\ 
                        &= \frac{1}{\det(\mathcal{C}^\alpha - \omega^2 \Id)}
                        \left(\begin{pmatrix}
                            a+ib & 0 \\
                            0 & a-ib
                        \end{pmatrix}\begin{pmatrix}
                            C_{11}(\alpha) & C_{12}(\alpha) \\
                            \overline{C_{12}(\alpha)} & C_{11}(\alpha)
                        \end{pmatrix}\begin{pmatrix}
                            (a-ib)C_{11}(\alpha)-\omega^2 & -(a+ib)C_{12}(\alpha) \\
                            -\overline{(a+ib)C_{12}(\alpha)} & (a+ib)C_{11}(\alpha)-\omega^2
                        \end{pmatrix}\right)_{11}\\
                        &= \frac{(a+ib)}{\det(\mathcal{C}^\alpha - \omega^2 \Id)}
                        (
                            -C_{11}(\alpha)\omega^2 + ((a-ib)C_{11}(\alpha)^2-(a-ib)\lvert C_{12}(\alpha)\rvert^2 )
                        )\\
                        &= \frac{(a+ib)}{\det(\mathcal{C}^\alpha - \omega^2 \Id)}
                        (-C_{11}(\alpha)\omega^2 + (a-ib)\det(C^\alpha))\\
                        &= \frac{(a+ib)(-C_{11}(\alpha)\omega^2 + (a-ib)\det(C^\alpha))}{
                            \omega^4 - 2a(C^\alpha)_{11}\omega^2 + \lvert(a+ib)\rvert^2 \det(C^\alpha)
                            }.
                        \\
                    \end{align*}

                Putting this into the determinant formula (\ref{eq11}), 
                we obtain that
                    \begin{align*}
                        0 &= 1 + \frac{-2ib}{a+ib}\frac{1}{\lvert Y^*\rvert}\int_{Y^*}\mathcal{C}^\alpha(\mathcal{C}^\alpha - \omega^2 \Id)^{-1}d\alpha\\
                        &= 1 + \frac{-2ib}{a+ib}\frac{1}{\lvert Y^*\rvert}\int_{Y^*}\frac{(a+ib)(-C_{11}(\alpha)\omega^2 + (a-ib)\det(C^\alpha))}{
                            \omega^4 - 2a(C^\alpha)_{11}\omega^2 + \lvert(a+ib)\rvert^2 \det(C^\alpha)
                            }d\alpha\\
                        &= \int_{Y^*}\frac{
                            \omega^4 - 2(a - ib)(C^\alpha)_{11}\omega^2 + (a-ib)^2\det(C^\alpha)
                            }{
                            \omega^4 - 2a(C^\alpha)_{11}\omega^2 + \lvert(a+ib)\rvert^2 \det(C^\alpha)
                                }d\alpha. \qedhere
                    \end{align*}

            \end{proof}

            \begin{figure}[h]
                \includegraphics[width = 0.49\textwidth]{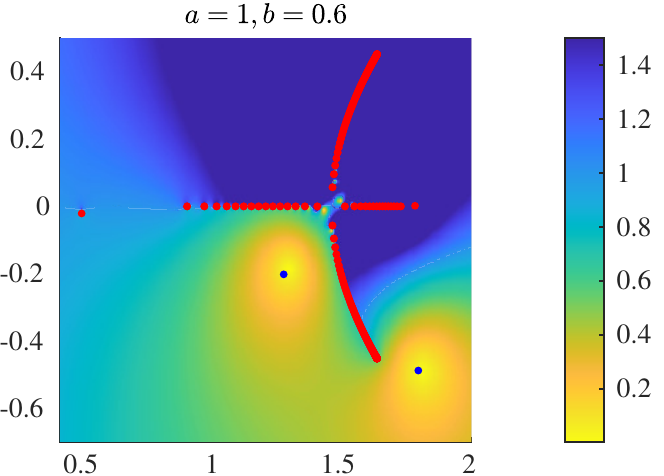}
                \includegraphics[width = 0.49\textwidth]{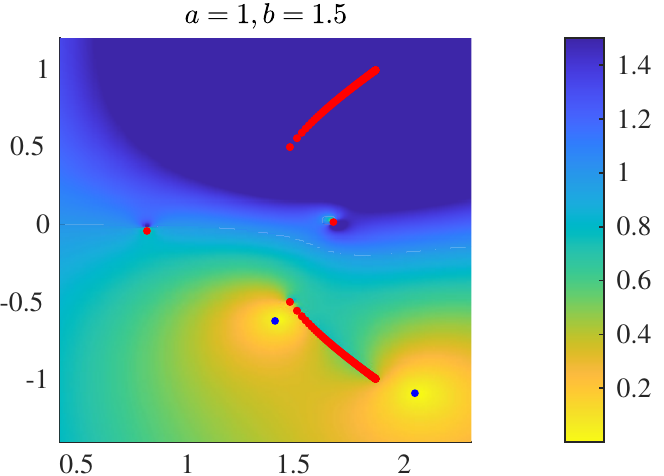}
                \caption{As an application and with the notation from Lemma \ref{lem:formula_defect_freq_PT} two simulations of two different $\mathcal{PT}$-symmetric materials with loss defect are depicted. The geometric structure is described in Figure \ref{fig:def_structure_loss}. The heatmap shows the absolute value of the right-hand side of equation \eqref{eq:omega_def_mode}, where all values bigger than $1.5$ are plotted in dark blue. In red is indicated a discretization of the continuous spectrum of the infinite material (199 discretization steps). This discretization of the Brillouin zone was then used to compute the heatmap. In blue is indicated an approximation of the defect mode eigenfrequency, when using a finite material with 200 fundamental domains to approximate the defected infinite material.}
                \label{fig:simulation_PT_loss_defect}
            \end{figure}
        
            Simulations of two different $\mathcal{PT}$-symmetric materials with loss defect are depicted in Figure \ref{fig:simulation_PT_loss_defect}. In this figure, the predicted defect mode eigenfrequencies of a truncated finite structure is also depicted and one sees that a finite structure is a good approximation for the exact prediction in the infinite structure. This gives an example where the spectral approximation theory developed in \cite{finite_infinite} for finite but large hermitian chains of resonators applies to the non-hermitian case. 

            All numerical results in this paper are obtained by the following way. The generalized quasiperiodic capacitance matrix is computed by using a multipole approximation for the single layer potential \cite[Appendix C]{bandgap}. The remaining numerical computations mainly rely on evaluations of the integrals, which occur in the different characterizing equations. These are computed using the end point rule with constant step sizes.

            The following observation will transform Lemma \ref{lem:formula_defect_freq_PT} into the corresponding lemma for a gain defect.
            Conjugation of the material parameters gives a bijection between the set of $\mathcal{PT}$-symmetric materials with loss defect and the set of $\mathcal{PT}$-symmetric materials with gain defect. Thus, conjugating the material parameters in the equation of Lemma \ref{lem:formula_defect_freq_PT} transforms the equation (\ref{eq11}) for loss defects into the corresponding equation for gain defects, proving  the following lemma.

            \begin{lemma}
                Let $D_1, D_2 \subset [0,1)\times \mathbb{R}^2$ be two equally sized spherical resonators, which are repeated periodically with $\Lambda = (1,0,0)^T\mathbb{Z}$ and let the $\Lambda$-periodic material parameters be given by
                    \begin{align*}
                        \delta_j (v_j)^2 = \begin{cases}
                            a - ib, &\text{if } j = 1,\\
                            a + ib, &\text{if } j = 2.
                        \end{cases}
                    \end{align*}
                Assume that the resonator $D_1^0$ is defected and has its material parameter given by $\overline{\delta_1(v_1)^2} = a+ib$. Then, the frequencies  $\omega$ of the defect modes in this setting are given by the solutions to the following equation 
                    \begin{align*}
                        0 = \int_{Y^*}\frac{
                        \omega^4 - 2(a + ib)(C^\alpha)_{11}\omega^2 + (a+ib)^2\det(C^\alpha)
                        }{
                        \omega^4 - 2a(C^\alpha)_{11}\omega^2 + \lvert(a+ib)\rvert^2 \det(C^\alpha)
                        }d\alpha.
                    \end{align*}
            \end{lemma}

        \subsubsection{General periodic system of resonator 
        dimers with a single or a double defect}\label{subsec:material_double_and_single_defect}
            In this section, we drop the $\mathcal{PT}$-symmetry assumption on the periodic material without defect, but treat the case of general, disjoint resonators $D_1, D_2 \subset [0,1) \times \mathbb{R}^2$ with general, periodic material parameters $V_1$ and $V_2$, respectively. Furthermore, any constraints on the defect material parameters are lifted, allowing, in the next section, to design defected materials with prescribed defect mode eigenfrequencies.

            \begin{figure}[h]
                \centering
                \includegraphics{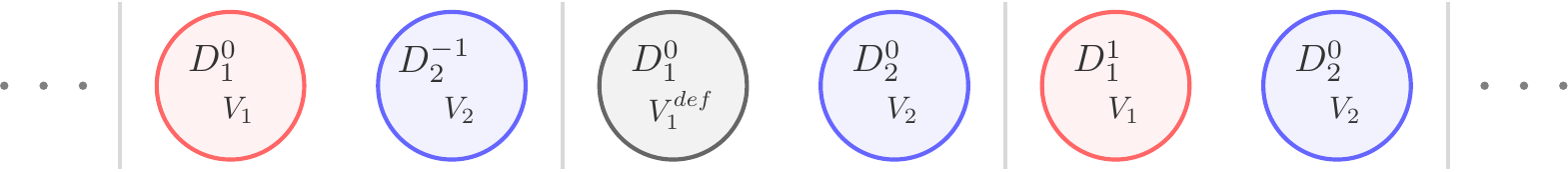}
                \caption{Notation convention of Lemma \ref{lem:formula_defect_freq_single}.}
                \label{fig:def_structure_with_defect_simplified}
            \end{figure}

            \begin{figure}
                \centering
                    \centering
                    \includegraphics[width=0.49\textwidth]{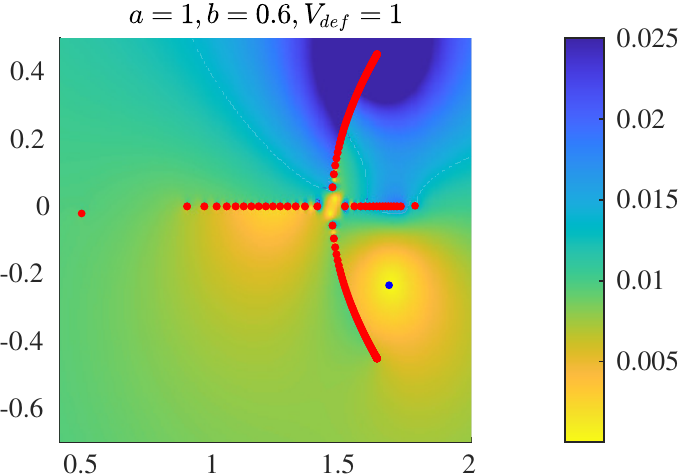}
                    \includegraphics[width=0.49\textwidth]{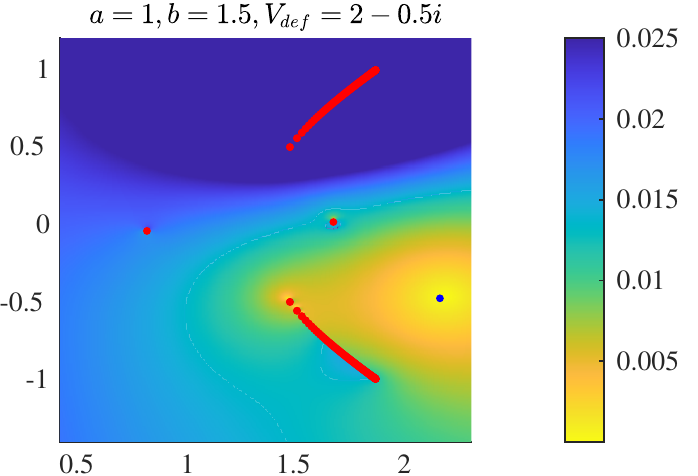}
                    \caption{As an application example of Lemma \ref{lem:formula_defect_freq_single}, two simulations of two single defects are depicted. The radius of the spherical resonators is given by $R= 0.15$, the centers of the resonators $D_1$ and $D_2$ are located at $(0.25,0,0)^T$ and $(0.75,0,0)^T$, respectively. The periodicity lattice used is $\Lambda = (1,0,0)^T\mathbb{Z}$. The periodic material parameters are given by $V_1 = a+ib$ and $V_2 = a-ib$. The heatmap shows the absolute value of the right-hand side of equation \eqref{eq:gen_single_def}, where all values bigger than $0.025$ are plotted in dark blue. In red is indicated a discretization of the continuous spectrum of the infinite material (199 discretization steps). This discretization of the Brillouin zone was then used to compute the heatmap. In blue is indicated an approximation of the defect mode eigenfrequency, when using a finite material with 200 fundamental domains to approximate the defected infinite material.}
            \end{figure}

            Considering first a single defected resonator, we obtain the following result. 

            \begin{lemma}\label{lem:formula_defect_freq_single}
                Let $D_1, D_2 \subset [0,1) \times \mathbb{R}^2$ be a pair of disjoint resonators which is repeated periodically with respect to $\Lambda = (1,0,0)^T\mathbb{Z}$. Set the material parameters as
                    \begin{align*}
                        \delta_1(v_1)^2 =: V_1 \in \mathbb{C},\\
                        \delta_2(v_2)^2 =: V_2 \in \mathbb{C},
                    \end{align*}
                as depicted in Figure \ref{fig:def_structure_with_defect_simplified}. Denote by $Y^* := \mathbb{R}/2\pi\mathbb{Z} \times \{0\}^2$ the associated Brillouin zone and by $\mathcal{C^\alpha}$ the associated capacitance matrix. Assume that the resonator $D_1$ of the zeroth cell is defected and its material parameter is given by $V_1^{def} \in \mathbb{C}$. Then, $\omega \in \mathbb{C} \setminus \cup_{\alpha \in Y^*}\sigma(\mathcal{C}^\alpha)$ is a defect mode eigenfrequency of the Helmholtz problem \eqref{eq:Helmholtz_prob} if and only if it satisfies
                    \begin{align}\label{eq:gen_single_def}
                        0 = V_1 + (V_1^{def} - V_1)\frac{1}{\lvert Y^*\rvert}\int_{Y^*}\frac{\det(\mathcal{C}^\alpha) - \omega^2\mathcal{C}^\alpha_{11}}{\det(\mathcal{C}^\alpha - \omega^2\Id)}d\alpha.
                    \end{align}
            \end{lemma}

            \begin{proof}
                Again, from \Cref{thm:defect}, we know that whenever $\omega \in \mathbb{C} \setminus \cup_{\alpha \in Y^*}\sigma(\mathcal{C}^\alpha)$, then $\omega$ corresponds to a defect mode eigenfrequency if and only if it satisfies the equation
                            \begin{align*}
                                \det(\Id - \mathcal{X}\mathcal{T}(\omega)) = 0.
                            \end{align*}
                        Since only one resonator is modified, this reduces to the equation
                            \begin{align*}
                                1 - X_0^1(\mathcal{T})_{11} = 0.
                            \end{align*}
                        Then we compute
                            \begin{align*}
                                (\mathcal{C}^\alpha(\mathcal{C}^\alpha - \omega^2 \Id)^{-1})_{11} &= \left(\mathcal{C}^\alpha\frac{1}
                                {\det(\mathcal{C}^\alpha - \omega^2\Id)}(\det(\mathcal{C}^\alpha)(\mathcal{C}^\alpha)^{-1} - \omega^2\Id)\right)_{11}\\
                                &= \left(\frac{1}{\det(\mathcal{C}^\alpha - \omega^2\Id)}(\det(\mathcal{C}^\alpha)\Id - \omega^2\mathcal{C}^\alpha)\right)_{11}\\
                                &= \frac{\det(\mathcal{C}^\alpha) - \omega^2\mathcal{C}^\alpha_{11}}{\det(\mathcal{C}^\alpha - \omega^2\Id)},
                            \end{align*}
                        and put this into the determinant formula to arrive at
                        \begin{align*}
                            0 &= 1 + X_0^1\frac{1}{\lvert Y^*\rvert}\int_{Y^*}\frac{\det(\mathcal{C}^\alpha) - \omega^2\mathcal{C}^\alpha_{11}}{\det(\mathcal{C}^\alpha - \omega^2\Id)}d\alpha\\
                            &= 1 + (V_1^{def}/V_1 - 1)\frac{1}{\lvert Y^*\rvert}\int_{Y^*}\frac{\det(\mathcal{C}^\alpha) - \omega^2\mathcal{C}^\alpha_{11}}{\det(\mathcal{C}^\alpha - \omega^2\Id)}d\alpha,
                        \end{align*}
                        which is equivalent to
                        \begin{align*}
                            0 = V_1 + (V_1^{def} - V_1)\frac{1}{\lvert Y^*\rvert}\int_{Y^*}\frac{\det(\mathcal{C}^\alpha) - \omega^2\mathcal{C}^\alpha_{11}}{\det(\mathcal{C}^\alpha - \omega^2\Id)}d\alpha,
                        \end{align*}
                        as desired.
            \end{proof}

            Similarly, in the setting of two defected resonators which lie in the same fundamental domain, we obtain the following characterization of the defect mode eigenfrequencies. 

            \begin{lemma}\label{lem:formula_defect_freq_double}
                Let $D_1, D_2 \subset [0,1) \times \mathbb{R}^2$ be a pair of disjoint resonators which are repeated periodically with respect to $\Lambda = (1,0,0)^T\mathbb{Z}$. Set the material parameters as
                    \begin{align*}
                        \delta_1(v_1)^2 =: V_1 \in \mathbb{C},\\
                        \delta_2(v_2)^2 =: V_2 \in \mathbb{C},
                    \end{align*}
                and denote by $Y^* := \mathbb{R}/2\pi\mathbb{Z} \times \{0\}^2$ the associated Brillouin zone and by $\mathcal{C^\alpha}$ the associated quasiperiodic generalized capacitance matrix. Assume that the zeroth cell is defected and its material parameters are given by $V_1^{def}$ and $V_2^{def}$. Then, $\omega \in \mathbb{C} \setminus \cup_{\alpha \in Y^*}\sigma(\mathcal{C}^\alpha)$ is a defect mode eigenfrequency of the Helmholtz problem \eqref{eq:Helmholtz_prob} if and only if it satisfies 
                    \begin{align*}
                        0 = \det\left(V + (V^{def} - V)\frac{1}{\abs{Y^*}}\int_{Y^*}\frac{1}{\det(\mathcal{C}^\alpha - \omega^2\Id)}\left[
                            \det(\mathcal{C}^\alpha)\Id - \omega^2 \mathcal{C}^\alpha
                        \right]d\alpha\right),
                    \end{align*}
                where
                    \begin{align*}
                        V := \begin{pmatrix}
                            V_1 & 0\\
                            0 & V_2
                        \end{pmatrix}, \quad V^{def} := \begin{pmatrix}
                            V_1^{def} & 0\\
                            0 & V_2^{def}
                        \end{pmatrix}.
                    \end{align*}
            \end{lemma}

            \begin{proof}
                Applying the formula for the inverse of a $2 \times 2$-matrix to the matrix $(\mathcal{C}^\alpha - \omega^2\Id)^{-1}$, one obtains that 
                    \begin{align*}
                        (\mathcal{C}^\alpha - \omega^2\Id)^{-1} &= \frac{1}{\det(\mathcal{C}^\alpha - \omega^2 \Id)}\left[\begin{pmatrix}
                            V_2\mathcal{C}^\alpha_{22} &V_1\mathcal{C}^\alpha_{12} \\
                        -V_2\mathcal{C}^\alpha_{21} &V_1\mathcal{C}^\alpha_{11}
                    \end{pmatrix} - \omega^2 \Id \right] \\
                    &= \frac{1}{\det(\mathcal{C}^\alpha - \omega^2\Id)}\left[\det(\mathcal{C}^\alpha)(\mathcal{C}^\alpha)^{-1} - \omega^2 \Id \right].
                    \end{align*}
                Putting this into the equation 
                    \begin{align*}
                        0 = \det\left(\Id - (V^{def}V^{-1}- \Id)\frac{1}{\abs{Y^*}}\int_{Y^*}\mathcal{C}^\alpha (\mathcal{C}^\alpha - \omega^2\Id)^{-1}d\alpha\right),
                    \end{align*}
                we get
                    \begin{align*}
                        0 = \det\left(\Id - (V^{def}V^{-1} - \Id)\frac{1}{\abs{Y^*}}\int_{Y^*}\frac{1}{\det(\mathcal{C}^\alpha - \omega^2\Id)}\left[
                            \det(\mathcal{C}^\alpha)\Id - \omega^2 \mathcal{C}^\alpha
                        \right]d\alpha\right).
                    \end{align*} Then pulling out $V$ yields
                    \begin{align*}
                        0 = \det\left(V - (V^{def} - V)\frac{1}{\abs{Y^*}}\int_{Y^*}\frac{1}{\det(\mathcal{C}^\alpha - \omega^2\Id)}\left[
                            \det(\mathcal{C}^\alpha)\Id - \omega^2 \mathcal{C}^\alpha
                        \right]d\alpha\right),
                    \end{align*}
                    which proves the lemma.
            \end{proof}

        \begin{remark}
            The results in Subsection \ref{subsec:material_double_and_single_defect} do not rely on the fact that a \emph{chain} of dimers is considered. That is, in the above two lemmas, Lemmas \ref{lem:formula_defect_freq_single} and \ref{lem:formula_defect_freq_double}, the sublattice $\Lambda = (1,0,0)^T\mathbb{Z}$ can be replaced with any $1,2$ or $3$-dimensional sublattice of $\mathbb{R}^3$ without changing the results and the characterizing equations.
        \end{remark}

        Both results will help us in the next section to design defected materials with defect modes at given defect mode eigenfrequencies.
        
    \subsection{Material design of defected resonator dimer systems with prescribed defect mode eigenfrequencies}\label{sec:Relation_between_defect_properties_and_localized_frequency}

        This section applies the characterization of defect mode eigenfrequencies in defected periodic materials to the question  whether it is possible to create defects which admit localized modes at a \emph{given} defect mode eigenfrequency $\omega$. It turns out that the answer to this question is, under mild assumptions, `yes'; see Theorem \ref{thm:formula_defect_freq}. Even more is true. In fact, it is possible to create a double defect that admits two localized modes $u_1$ and $u_2$ at given defect mode eigenfrequencies $\omega_1$ and $\omega_2$, as it is shown in the proof of Theorem \ref{thm:formula_double_defect_freq}. We conclude with a remark on realizing $n$ frequencies as defect mode eigenfrequencies in a material with $n$ defects.

        \begin{theorem}\label{thm:formula_defect_freq}
            Let $D_1, D_2 \subset [0,1) \times \mathbb{R}^2$ be a pair of disjoint resonators which is repeated periodically with respect to $\Lambda = (1,0,0)^T\mathbb{Z}$. Set the material parameters as (see e.g. Figure \ref{fig:def_structure_with_defect_simplified})
            \begin{align*}
                \delta_1(v_1)^2 =: V_1 \in \mathbb{C},\\
                \delta_2(v_2)^2 =: V_2 \in \mathbb{C},
            \end{align*}
            and denote by $Y^* := \mathbb{R}/2\pi\mathbb{Z} \times \{0\}^2$ the associated Brillouin zone and by $\mathcal{C^\alpha}$ the associated quasiperiodic generalized capacitance matrix.
            Let $W$ be the set of zeros of $\int_{Y^*}\frac{\det(\mathcal{C}^\alpha) - \omega^2\mathcal{C}^\alpha_{11}}{\det(\mathcal{C}^\alpha - \omega^2\Id)}d\alpha$, that is, $$ W := \left\{\omega \in \mathbb
            {C}\setminus \bigcup_{\alpha\in Y^*}\sigma(\mathcal{C}^\alpha) : \int_{Y^*}\frac{\det(\mathcal{C}^\alpha) - \omega^2\mathcal{C}^\alpha_{11}}{\det(\mathcal{C}^\alpha - \omega^2\Id)}d\alpha = 0\right\}.$$
            Then, the following map $\Psi: \mathbb{C}\setminus (\bigcup_{\alpha\in Y^*}\sigma(\mathcal{C}^\alpha)\cup W) \rightarrow \mathbb{C}$ associates to each $\omega$ a choice of defect material parameters $V^{def}(\omega)$ (for the resonator $D_1^0$) such that $\omega$ occurs as a defect mode eigenfrequency, when in the above periodic structure a defected fundamental domain is created with the defect material parameter given by $V^{def}(\omega)$. The map $\Psi$ is defined by 
            \begin{align*}
                \begin{matrix}
                    \Psi:\quad &\mathbb{C}\setminus \left(\bigcup_{\alpha\in Y^*}\sigma(\mathcal{C}^\alpha)\cup W\right) &\longrightarrow &\mathbb{C}\\
                    \\
                &\omega &\longmapsto &V^{def}(\omega),
                \end{matrix}
            \end{align*}
            where $V^{def}(\omega)$ is given by the following
            \begin{align*}
                V^{def}(\omega) = V_1\left(1 - \frac{1}{\frac{1}{\lvert Y^*\rvert}\int_{Y^*}\frac{\det(\mathcal{C}^\alpha) - \omega^2\mathcal{C}^\alpha_{11}}{\det(\mathcal{C}^\alpha - \omega^2\Id)}d\alpha}\right).
            \end{align*}
        \end{theorem}

        \begin{proof}
            This result directly follows from Lemma \ref{lem:formula_defect_freq_single} and is obtained by a straightforward computation from equation \eqref{eq:gen_single_def}.
        \end{proof}
        
        \begin{figure}
            \centering
            \begin{subfigure}[t]{0.49\textwidth}
                \centering
                \includegraphics[width=\textwidth]{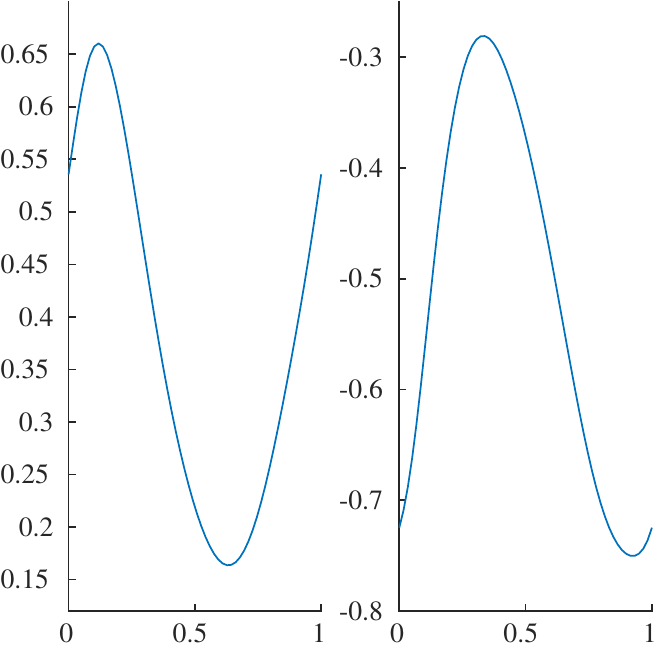}
                \caption{Plot of the map $x \mapsto V^{def}(\omega(x))$, where $\omega(x) = 1 - 0.4i + 0.2\exp(2\pi i x)$. On the left is the real  part of $V^{def}(\omega(x))$ and on the right is the imaginary part.}
            \end{subfigure}
            \hfill
            \begin{subfigure}[t]{0.49\textwidth}
                \centering
                \includegraphics[width=\textwidth]{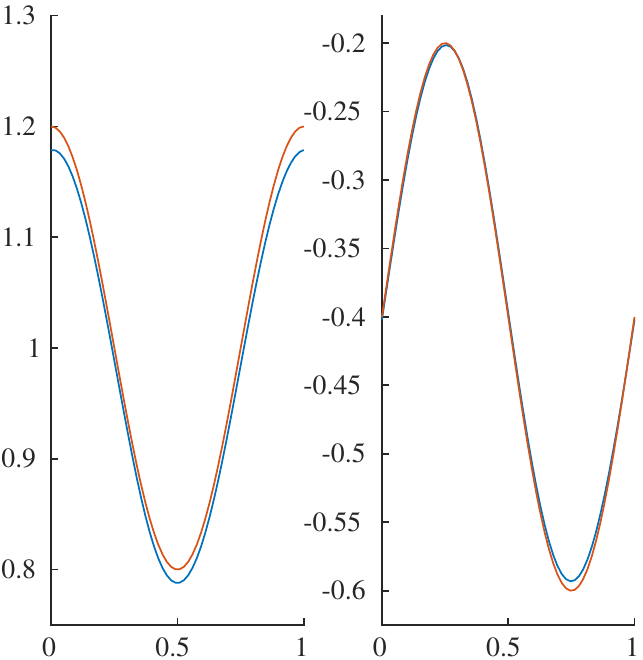}
                \caption{Comparison between $\omega(x) = 1 - 0.4i + 0.2\exp(2\pi i x)$ (blue) and the approximated defect mode eigenfrequency (red) obtained by a finite approximation of the material with defect $V^{def}(\omega(x))$ and $100$ fundamental domains.}
            \end{subfigure}
            \caption{Material design of a defected chain of resonator dimers that admits $\omega$ as a defect mode eigenfrequency. With the notation from Theorem \ref{thm:formula_defect_freq}, the following material parameters are chosen as follows: $V_1 = 1+0.6i, V_2 = 1.0-0.6i$, the radius of all the spherical resonators is given by $R = 0.15$ and the locations of the centers of the two resonators $D_1, D_2$ in the fundamental domain $[0,1) \times \mathbb{R}^2$ are given by $(0.25,0,0)$ and $(0.75,0,0)$, respectively. The considered defect mode eigenfrequencies are given by $\{ \omega =  1 - 0.4i + 0.2\exp(2\pi i x) : x \in [0,1) \}$. In this figure, the $x$-axis of all the plots is parameterized by $x \in [0,1)$. }
            \label{fig:material_design_single_defect}
        \end{figure}
        
        An illustration of the result stated in Theorem \ref{thm:formula_defect_freq} can be found in Figure \ref{fig:material_design_single_defect}. There, the map $\omega \mapsto V^{def}(\omega)$ is plotted and the resulting values $V^{def}(\omega)$ are validated by  a comparison between the desired defect mode eigenfrequency $\omega$ and the defect mode eigenfrequency $\omega_{approx}$ obtained when one approximates the defected material with a truncated, finite material.

        Similarly, given two frequencies $\omega_1$ and $\omega_2$, we succeed to produce defect parameters $V^{def}_1, V^{def}_2$ of a double defect, such that $\omega_1$ and $\omega_2$ are the eigenfrequencies of two localized modes $u_1, u_2$. The precise result is stated in the following theorem.

        \begin{theorem}\label{thm:formula_double_defect_freq}
            Let $D_1, D_2 \subset [0,1) \times \mathbb{R}^2$ be a pair of disjoint resonators which are repeated periodically with respect to $\Lambda = (1,0,0)^T\mathbb{Z}$. Set the material parameters as
                \begin{align*}
                    \delta_1(v_1)^2 =: V_1 \in \mathbb{C},\\
                    \delta_2(v_2)^2 =: V_2 \in \mathbb{C},
                \end{align*}
            and denote by $Y^* := \mathbb{R}/2\pi\mathbb{Z} \times \{0\}^2$ the associated Brillouin zone and by $\mathcal{C^\alpha}$ the associated quasiperiodic generalized capacitance matrix. Then, the following map $\Psi: \left(\mathbb{C}\setminus \bigcup_{\alpha \in Y^*}\sigma(\mathcal{C}^\alpha)\right)^2\setminus W  \rightarrow \mathbb{C}^2$ associates to each pair of frequencies $(\omega_1,\omega_2)$ a choice of defect material parameters $(V_1^{def}(\omega_1,\omega_2),V_2^{def}(\omega_1,\omega_2))$ such that $\omega_1$ and $\omega_2$ occur as defect mode eigenfrequencies, when in the above material a defected fundamental domain is created with the defect material parameters given by $V_1^{def}(\omega_1,\omega_2)$ and $V_2^{def}(\omega_1,\omega_2)$ for the two resonators $D^0_1$ and $D^0_2$, respectively. The map $\Psi$ is given by 
            \begin{align*}
            \begin{matrix}
                \Psi:\quad &\left(\mathbb{C}\setminus \bigcup_{\alpha \in Y^*}\sigma(\mathcal{C}^\alpha)\right)^2\setminus W &\longrightarrow &\hspace{30pt}\mathbb{C}^2\\
                \\
                        &(\omega_1,\omega_2) &\longmapsto &\hspace{30pt}(V_1^{def}(\omega_1,\omega_2),V_2^{def}(\omega_1,\omega_2))
            \end{matrix}
            \end{align*}
            where $(V_1^{def}(\omega_1,\omega_2),V_2^{def}(\omega_1,\omega_2))$ and $W$ are given by the following.
            Let $I_{11}(\omega),I_{12}(\omega),I_{21}(\omega),I_{22}(\omega)$ be defined as follows:
            \begin{align*}
                I_{11}(\omega) &= \frac{1}{\abs{Y^*}}\int_{Y^*} \frac{\det\left(\mathcal{C}^\alpha\right)-\omega \mathcal{C}^\alpha_{11}}{\det\left(\mathcal{C}^\alpha - \omega^2\Id\right)}d\alpha,\qquad
                I_{12}(\omega) = \frac{1}{\abs{Y^*}}\int_{Y^*} \frac{-\omega \mathcal{C}^\alpha_{12}}{\det\left(\mathcal{C}^\alpha - \omega^2\Id \right)}d\alpha,\\
                \\
                I_{21}(\omega) &= \frac{1}{\abs{Y^*}}\int_{Y^*} \frac{-\omega \mathcal{C}^\alpha_{21}}{\det\left(\mathcal{C}^\alpha - \omega^2\Id \right)}d\alpha, \qquad
                I_{22}(\omega) = \frac{1}{\abs{Y^*}}\int_{Y^*} \frac{\det\left(\mathcal{C}^\alpha\right)-\omega \mathcal{C}^\alpha_{22}}{\det\left(\mathcal{C}^\alpha - \omega^2\Id \right)}d\alpha, \\
                \\
                & \qquad \qquad \qquad \qquad I(\omega) = \begin{pmatrix}
                    I_{11}(\omega) &I_{12}(\omega)\\
                    I_{21}(\omega) &I_{22}(\omega)
                \end{pmatrix}.
            \end{align*}
            Then, $W$ is given by
            \begin{align*}
                W = \left\{ (\omega_1,\omega_2) \in \left(\mathbb{C}\setminus \bigcup_{\alpha \in Y^*}\sigma(\mathcal{C}^\alpha)\right)^2 : \begin{matrix} I_{11}(\omega_1)\det(I(\omega_2))-I_{11}(\omega_2)\det(I(\omega_1)) = 0\\ \text{ or }\\ (V_1^{def}(\omega_1,\omega_2)-V_1)\det(I(\omega)) + V_1I_{22}(\omega) = 0\end{matrix}
                \right\}
            \end{align*}
            and $V_1^{def}(\omega_1,\omega_2),V_2^{def}(\omega_1,\omega_2)$ are given by
            \begin{align*}
                V_1^{def}(\omega_1,\omega_2) = V_1-\frac{1}{2}p \pm \sqrt{\frac{p^2}{4} - q}, \quad
                V_2^{def}(\omega_1,\omega_2) = V_2\left(1 - \frac{V_1 + I_{11}(\omega)(V_1^{def}-V_1)}{(V_1^{def}-V_1)\det(I(\omega)) + V_1I_{22}(\omega)}\right)
            \end{align*}
            with
            \begin{align*}
                p& = \frac{\det(I(\omega_1))-\det(I(\omega_2)) + I_{22}(\omega_1)I_{11}(\omega_2)-I_{22}(\omega_2)I_{11}(\omega_1)}{I_{11}(\omega_2)\det(I(\omega_1))-I_{11}(\omega_1)\det(I(\omega_2))}, \\
                q &= \frac{I_{22}(\omega_1)-I_{22}(\omega_2)}{I_{11}(\omega_2)\det(I(\omega_1))-I_{11}(\omega_1)\det(I(\omega_2))}.
            \end{align*}
        \end{theorem}

        \begin{proof}
            From Lemma \ref{lem:formula_defect_freq_double}, we know that $\omega$ is a defect mode eigenfrequency if 
                \begin{align*}
                    0 = \det\left(V + (V^{def} - V)\frac{1}{\abs{Y^*}}\int_{Y^*}\frac{1}{\det(\mathcal{C}^\alpha - \omega^2\Id )}\left[
                        \det(\mathcal{C}^\alpha)\Id - \omega^2 \mathcal{C}^\alpha
                    \right]d\alpha\right),
                \end{align*}
            where
                \begin{align*}
                    V := \begin{pmatrix}
                        V_1  & 0\\
                        0 & V_2
                    \end{pmatrix}, \quad V^{def} := \begin{pmatrix}
                        V_1^{def} & 0\\
                        0 & V_2^{def}
                    \end{pmatrix}.
                \end{align*}
            That is, if the following equation holds for $\omega = \omega_1$ and $\omega = \omega_2$,
                \begin{align*}
                    0   &= \det\left(V + (V^{def} - V)I(\omega)\right),\\
                        &= (V_1 + (V^{def}_1 - V_1)I_{11}(\omega))(V_2 + (V^{def}_2 - V_2)I_{22}(\omega)) - (V^{def}_1 - V_1)(V^{def}_2 - V_2)I_{12}(\omega)I_{21}(\omega) \\
                        &= X_1X_2\det(I(\omega)) + X_1V_2I_{11}(\omega) + X_2V_1I_{22}(\omega) + V_1V_2, 
                \end{align*}
            where $X_i = V^{def}_i - V_i$ with $i = 1,2$, then $\omega_1$ and $\omega_2$ are defect mode eigenfrequencies, simultaneously.
            The function $\Psi$ is obtained by deducing from the above equation the parameter $X_2$ as a function of $\omega$ and then equating $X_2(\omega_1)=X_2(\omega_2)$ to obtain an equation for $X_1$ (without $X_2$) from which an expression for $X_1(\omega_1)=X_1(\omega_2)$ can be deduced.
            
            One obtains
                \begin{align*}
                    X_2(\omega) = -\frac{V_2(V_1 + I_{11}(\omega)X_1)}{X_1\det(I(\omega))+V_1I_{22}(\omega)},
                \end{align*}
    and therefore, equating $ X_2(\omega_1)=X_2(\omega_2)$ gives
                \begin{align*}
                    (V_1 + I_{11}(\omega_1)X_1)(X_1\det(I(\omega_2))+V_1I_{22}(\omega_2)) = (V_1 + I_{11}(\omega_2)X_1)(X_1\det(I(\omega_1))+V_1I_{22}(\omega_1)).
                \end{align*}
            This is a quadratic equation in $X_1$ which can be solved with  standard formulas to obtain the result of this theorem.
        \end{proof}

        \begin{figure}[ht]
            \centering
            \begin{subfigure}[t]{0.49\textwidth}
                \centering
                \includegraphics[width=\textwidth]{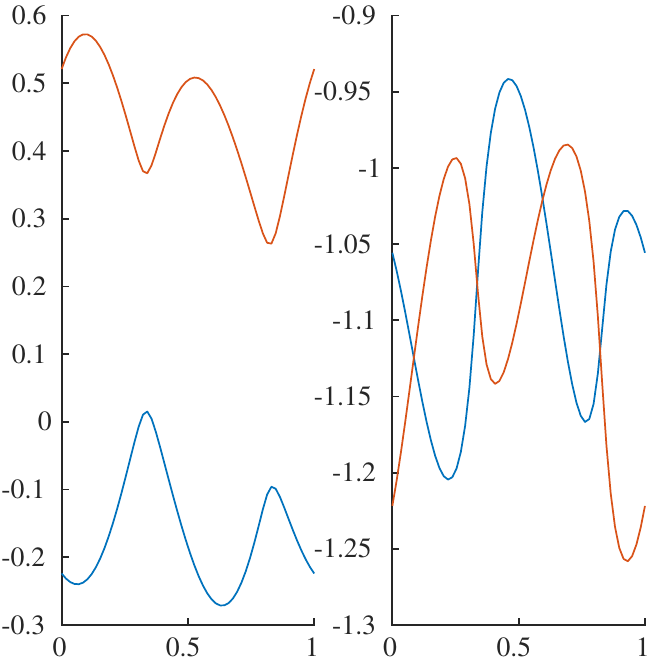}
                \caption{Plot of the map $x \longmapsto\Psi(\omega_1(x),\omega_2(x)) = (V_1^{def}(\omega_1(x),\omega_2(x)),V_2^{def}(\omega_1(x),\omega_2(x)))$, where $\omega_1(x) = 1 - 0.4i + 0.2\exp(2\pi i x)$ and $\omega_2(x) = 1 - 0.4i + 0.2\exp(2\pi i (x+\pi/2)-0.2)$. On the left are the real parts of $V_1^{def}(\omega_1(x),\omega_2(x)),V_2^{def}(\omega_1(x),\omega_2(x))$  in blue and red, respectively and on the right are the imaginary  parts in blue and red, respectively.  }
            \end{subfigure}
            \hfill
            \begin{subfigure}[t]{0.49\textwidth}
                \centering
                \includegraphics[width=\textwidth]{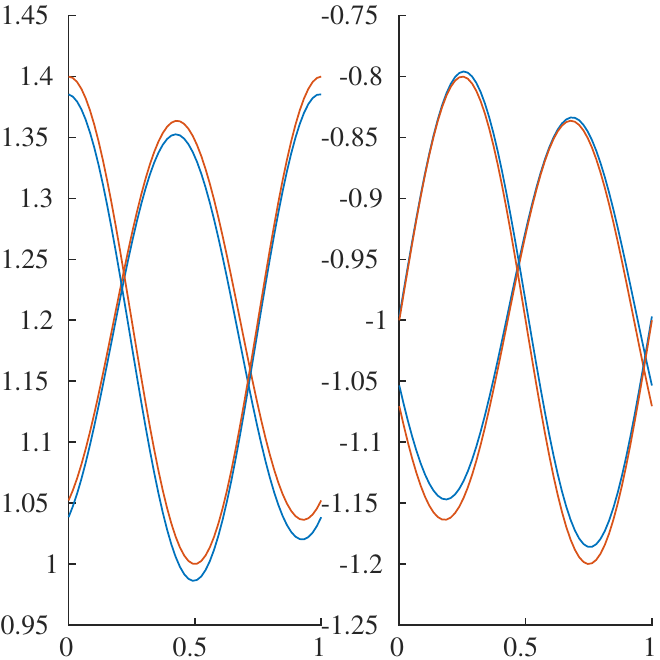}
                \caption{Comparison between $\omega_1(x),\omega_2(x)$ (red) and approximate defect mode eigenfrequencies (blue) associated to a finite approximation of the material with double defect $V_1^{def}(\omega_1(x),\omega_2(x)),V_2^{def}(\omega_1(x),\omega_2(x))$ and with $100$ fundamental domains. \vspace{11pt}}
            \end{subfigure}
            \caption{Material design of a chain of resonator dimers with a double defect that admits  $\omega_1,\omega_2$ as defect mode eigenfrequencies. With the notation from Theorem \ref{thm:formula_double_defect_freq}, the material parameters are chosen as follows: $V_1 = 1+0.6i, V_2 = 1.0-0.6i$, the radius of all the spherical resonators is given by $R = 0.15$ and the locations of the centers of the two resonators $D_1, D_2$ in the fundamental domain $[0,1) \times \mathbb{R}^2$ are given by $(0.25,0,0)$ and $(0.75,0,0)$, respectively. The considered defect mode eigenfrequencies are given by $\{ (\omega_1(x),\omega_2(x)) : x \in [0,1) \}$ with $\omega_1(x) =  1.2 - i + 0.2\exp(2\pi i x)$ and $\omega_2(x) = 1.2 - i + 0.2\exp(2\pi i (x + \pi/2) -0.2)$. In this figure, the $x$-axis of all the plots is parameterized by $x \in [0,1)$. }
            \label{fig:material_design_double_defect}
        \end{figure}

        \begin{remark}
            We would like to remark that the question whether $n$ given frequencies $(\omega_1,\ldots,\omega_n)$ can simultaneously occur as defect mode eigenfrequencies in a periodic structure with $n$ defects and what the defect material parameters would be, can be answered by solving $n$ polynomial equations for a common root. Namely, looking at equation \eqref{eq:general_eq_for_defect}, identities for the corresponding defect material parameters $V^{def}_1,\ldots,V^{def}_n$ are obtained by solving the system
                \begin{align}\label{eq:material_design_multiple_defect}
                    \det(I - X_0T^0(\omega_1)) = 0,~ \ldots,~ \det(I - X_0T^0(\omega_n)) = 0,
                \end{align}
            where $ X_0$ is a diagonal matrix with $(X_0)_{ii} = V^{def}_i - V_i$.\footnote{If the locations of the defects are not in the same fundamental domain of the lattice, then one needs to consider the associated submatrix of $\mathcal{T}(\omega)$ (see equation \eqref{eq:toeplitz_mat}). This particularly occurs when the number of defects is higher than the number of resonators in the fundamental domain. This adjustment in the system of equations \eqref{eq:material_design_multiple_defect} does however not change the polynomial nature of the problem and the reasoning in this remark remains valid.} This system is a set of $n$ polynomial equations in $n$ variables and as such, generically, has a solution in $\mathbb{C}^n$. This implies that $n$ generic frequencies can be realized as defect mode eigenfrequencies in a chain of resonator dimers with $n$ defects. Closed form solutions, as presented in Theorems \ref{thm:formula_defect_freq} and \ref{thm:formula_double_defect_freq}, would be possible for up to 4 frequencies. For more than 4 frequencies, the common roots need to be found numerically.
        \end{remark}

        \begin{remark}
            As in Subsection \ref{subsec:material_double_and_single_defect}, the results in  Subsection \ref{sec:Relation_between_defect_properties_and_localized_frequency} do not rely on the fact that a \emph{chain} of dimers is considered. That is, in the above two Theorems, Theorems \ref{thm:formula_defect_freq} and \ref{thm:formula_double_defect_freq}, the sublattice $\Lambda = (1,0,0)^T\mathbb{Z}$ can be replaced with any $1,2$ or $3$-dimensional sublattice of $\mathbb{R}^3$ without changing the results and the characterizing formulas.
        \end{remark}

\section{Instantly changing materials}\label{sec:Spatio-temporal_localization}

        The aim of this section is to develop a way to design instantly changing materials which admit prescribed \emph{time-dependent} defect mode eigenfrequencies. Furthermore, this will allow the design of materials with spatio-temporally localized modes. To this end, we will first analyze the existence of quasi-harmonic solutions in the presence of instantly changing material parameters. Then, these results are used to design materials with prescribed time-dependent defect mode eigenfrequency. As an application of this material design, materials which admit spatio-temporally localized modes will be treated.

        \begin{figure}[h]
            \centering
            \includegraphics{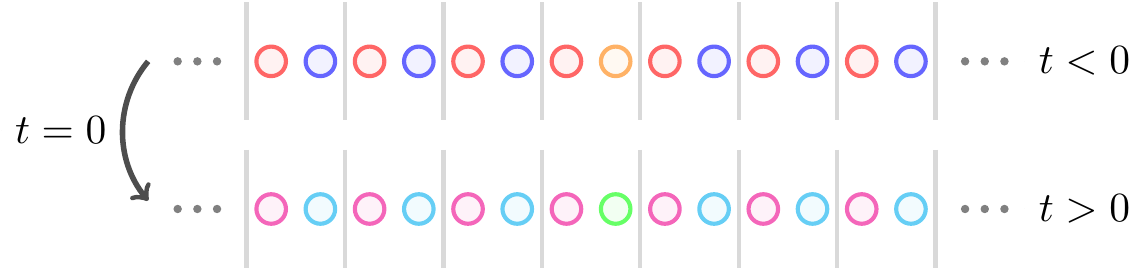}
            \caption{Illustration of the type of time-dependent material used in Section \ref{sec:Spatio-temporal_localization}.}
            \label{fig:def_time_dep_structure}
        \end{figure}

    \subsection{Quasi-harmonic waves in instantly changing materials}\label{subsec:Instantly_changing_material}

        Here, we consider the wave equation with space- and time-dependent materials parameters
            \begin{align}\label{eq:wave_eq}
                \left(\frac{\partial}{\partial t}\frac{1}{\kappa(x,t)}\frac{\partial}{\partial t} - \diver_x \frac{1}{\rho(x,t)}\nabla_x\right)u(x,t) = 0, \quad x \in \mathbb{R}^3, \quad t \in \mathbb{R},
            \end{align}
        with $\rho(x,t)$ and $\kappa(x,t)$ being of the form 
            \begin{align}\label{eq:time_dep_rho_kappa}
                \rho(x,t) = \begin{cases}
                    \rho^-(x) & \text{for } t < 0, \\
                    \rho^+(x) & \text{for } t \geq 0,
                \end{cases}
                \quad \text{ and } \quad
                \kappa(x,t) = \begin{cases}
                    \kappa^-(x) & \text{for } t < 0, \\
                    \kappa^+(x) & \text{for } t \geq 0,
                \end{cases}
            \end{align}
        where $\rho^-(x), \rho^+(x)$ and $\kappa^-(x), \kappa^+(x)$ are $\mathbb{C}^*$-valued functions on $\mathbb{R}^3$.
        In this setting, we are interested in what we call  \emph{quasi-harmonic} solutions.

        \begin{definition}[Quasi-harmonic modes]
            A non-zero solution $u(x,t)$ of the wave-equation \eqref{eq:wave_eq} is called \emph{quasi-harmonic} if it is of the form
                \begin{align}\label{eq:quasi_time_harm_sol}
                    u(x,t) = v(x)\exp(-i\omega(t)t),
                \end{align}
            where $v: \mathbb{R}^3 \rightarrow \mathbb{C}$ is independent of the time variable and $\omega(t)$ has the form
                \begin{align}\label{eq:time-dep_omega}
                    \omega(t) = \begin{cases}
                        \omega^- & \text{for } t < 0,\\
                        \omega^+ & \text{for } t > 0,
                    \end{cases}
                \end{align}
            for $\omega^-, \omega^+ \in \mathbb{C}$. In this case, $\omega(t)$ is called \emph{time-dependent frequency} of $u(x,t)$.
        \end{definition}
        
        We will see that quasi-harmonic solutions only occur if $\omega^+\rho^+ = \omega^-\rho^-$ and $\omega^-\kappa^+ = \omega^+\kappa^-$. That is, quasi-harmonic solutions only occur if $\rho$ and $\kappa$ satisfy a type of \emph{temporal Snell's law}
            $$\frac{\omega^+}{\omega^-} = \frac{\rho^-}{\rho^+} = \frac{\kappa^+}{\kappa^-},$$
        where $\omega$ takes the position of the angle of incidence, $\rho$ the refractive index and $\kappa$ the phase velocity.

        \begin{prop}[Temporal Snell's law]\label{prop:snell's_law}\label{prop:constant_density}
            In the above setting and under the assumption that a non-zero quasi-harmonic solution $u(x,t)$ with $\omega^+, \omega^- \not = 0$ exists, then the following relation is satisfied by $\omega^-, \omega^+$, $\rho^-,\rho^+$ and $\kappa^-,\kappa^+$:
                \begin{equation} \label{snelllaw}
                    \frac{\omega^+}{\omega^-} = \frac{\kappa^+}{\kappa^-} = \frac{\rho^-}{\rho^+}.
                \end{equation}
        \end{prop}

        \begin{remark}
            If $\omega^+ \not=0$ and $\omega^-\not=0$, then Snell's law in particular implies that ${\kappa^+}(x) = \frac{\omega^+}{\omega^-}\kappa^-(x)$ and ${\rho^+}(x) = \frac{\omega^-}{\omega^+}\rho^+(x)$ for all $x \in \mathbb{R}^3$, where $\frac{\omega^-}{\omega^+}$ is constant in space.
        \end{remark}

        \begin{proof}
            Assume $u(x,t) = v(x)\exp(-i\omega(t))$ is a non-zero quasi-harmonic solution with $\omega^-,\omega^+\not=0$.
            Since the time-derivative is continuous, it follows
                \begin{align*}
                    \lim_{t \nearrow 0} \frac{1}{\kappa(x,t)}\frac{\partial}{\partial t}u(x,t) = \lim_{t \searrow 0} \frac{1}{\kappa(x,t)}\frac{\partial}{\partial t}u(x,t),
                \end{align*}
            which in this case is simply given by
                \begin{align*}
                    \frac{\omega^-}{\kappa^-} = \frac{\omega^+}{\kappa^+}.
                \end{align*}

            Then, putting the ansatz \eqref{eq:quasi_time_harm_sol} into equation \eqref{eq:wave_eq} and considering the associated Helmholtz problem one obtains that, for $t<0$, 
            $$ \frac{(\omega^-)^2}{\kappa^-}v(x) + \frac{1}{\rho^-}v(x) = 0$$
            and for $t\geq0$
            $$ \frac{(\omega^+)^2}{\kappa^+}v(x) + \frac{1}{\rho^+}v(x) = 0.$$
            Since $\frac{\omega^-}{\kappa^-} = \frac{\omega^+}{\kappa^+}$ by the first equality, it thus follows that $\rho^+\omega^+ = \rho^-\omega^-$.
        \end{proof}

        Since the material parameters $\rho$ and $\kappa$ are given as in equation \eqref{eq:time_dep_rho_kappa} and since we are interested in quasi-harmonic solutions, it is possible to solve the wave equation \eqref{eq:wave_eq} by solving the following associated Helmholtz problems for negative and positive times $t$:
        \begin{align*}
                \Delta v + \omega^2\frac{\kappa^-(x)}{\rho^-(x)} v = 0 & \text{ for } v: \mathbb{R}^3 \rightarrow \mathbb{C} \text{ and } \omega \in \mathbb{C},\\ \text{ and }\\
                \Delta v + \omega^2\frac{\kappa^+(x)}{\rho^+(x)} v = 0 & \text{ for } v: \mathbb{R}^3 \rightarrow \mathbb{C} \text{ and } \omega \in \mathbb{C},
        \end{align*}
        and then using a solution $(v(x),\omega)$ for negative or positive times as an initial condition for positive or negative times, respectively. This approach leads to the following theorem.

        \begin{theorem}\label{thm:time_dependent}
            Consider the above setting and assume that a non-zero solution $v:\mathbb{R}^3\rightarrow \mathbb{C}$ of the Helmholtz problem
                \begin{align*}
                        \Delta v + (\omega^-)^2\frac{\kappa^-(x)}{\rho^-(x)} v = 0 & \text{ in } \mathbb{R}^3,
                \end{align*}
            exists. Then, $v(x)\exp(-i\omega^-t): \mathbb{R}^3\times\mathbb{R}_{<0} \rightarrow \mathbb{C}$ corresponds to the restriction of a quasi-harmonic solution $u(x,t)$ of the wave equation \eqref{eq:wave_eq} if and only if $\rho$ and $\kappa$ satisfy the temporal Snell law \eqref{snelllaw}.
        \end{theorem}

        \begin{remark}
            Equivalently, given a solution $(v(x),\omega^+)$ to the a Helmholtz problem
                \begin{align*}
                        \Delta v + (\omega^+)^2\frac{\kappa^+(x)}{\rho^+(x)} v = 0 & \text{ in } \mathbb{R}^3.
                \end{align*}
                Then, $v(x)\exp(-i\omega^+t): \mathbb{R}^3\times\mathbb{R}_{\geq0} \rightarrow \mathbb{C}$ corresponds to the restriction of a quasi-harmonic solution $u(x,t)$ of the wave equation \eqref{eq:wave_eq} if and only if $\rho$ and $\kappa$ satisfy the temporal Snell law \eqref{snelllaw}.
        \end{remark}
        
        \begin{proof}
            Assume that $v(x)\exp(-i\omega^-t): \mathbb{R}^3\times\mathbb{R}_{<0} \rightarrow \mathbb{C}$ is the restriction of a mode of the form $u(x,t) = v(x)\exp(-i\omega(t)t)$, where $\omega(t) = \omega^-$ for $t<0$ and $\omega(t) = \omega^+$ for $t\geq0$. Then  Propositions \ref{prop:snell's_law} implies that $\rho$ and $\kappa$ satisfy the temporal Snell law. Conversely, if $\rho$ and $\kappa$ are as above and satisfy the temporal Snell's law, then $u(x,t) = v(x)\exp(-i\omega(t)t)$ defines a solution to \eqref{eq:wave_eq} and its restriction to the negative time domain is given by $v(x)\exp(-i\omega^-t): \mathbb{R}^3\times\mathbb{R}_{<0} \rightarrow \mathbb{C}$.
        \end{proof}

        These general results allow to develop a material design theory for subwavelength quasi-harmonic solutions with prescribed time-dependent defect mode eigenfrequency in the next section.

    \subsection{Design of instantly changing defected dimer systems with prescribed time-dependent defect mode eigenfrequencies}\label{subsec:temporal_material}

        As in Subsection \ref{sec:Theoretical_background_and_setting}, we will consider disjoint resonators $D_1, D_2 \subset [0,1) \times \mathbb{R}^2$ repeated periodically with respect to the lattice $\Lambda = (1,0,0)^T\mathbb{Z}$ in $\mathbb{R}^3$, forming a chain of resonator dimers. The resonator domain is then given by
            $$\mathcal{D} := \bigcup_{m \in \mathbb{Z}} D^m_1 \cup D^m_2, \quad \text{ where }\quad D_i^m := \begin{pmatrix}
                m \\ 0 \\ 0
            \end{pmatrix} + D_i.$$

        Similarly as before, we associate to every resonator $D_i^m$ with $i=1,2$ and $m \in \mathbb{Z}$  the time-dependent material parameters $\rho_i^m(t)$ and $\kappa_i^m(t)$, and to the background medium the time-dependent material parameters $\rho_b(t)$ and $\kappa_b(t)$.
        Moreover, we suppose that the resonator material parameters are periodic with respect to the lattice $\Lambda$, except for the resonator $D^0_1$, which might have different material parameters. That is, we assume that
            \begin{align*}
                \kappa_i^m(t) = \kappa_i^n(t), \quad
                \rho_i^m(t) = \rho_i^n(t),
            \end{align*}
         for $m,n \in \mathbb{Z}$ and $(m,i),(n,i) \not=(0,1)$.
        In order to stay in the scope of Section \ref{subsec:Instantly_changing_material}, we assume that $\kappa_i^m(t), \kappa_b(t), \rho_i^m(t)$ and $\rho_b(t)$ are constant for negative times and constant for positive times, but may have a jump at time $t=0$. We will call time-dependent materials of this type \emph{instantly changing materials}. We will adopt the following notation:
            \begin{align*}
                \kappa(x,t) = \begin{cases}
                    \kappa_i^m(t) &\text{for } x\in D_i^m \text{ for some } m\in \mathbb{Z}, i = 1,2,\\
                    \kappa_b(t) &\text{for } x \not\in \mathcal{D},
                \end{cases} \\
                \rho(x,t) = \begin{cases}
                    \rho_i^m(t) &\text{for } x\in D_i^m \text{ for some } m\in \mathbb{Z}, i = 1,2,\\
                    \rho_b(t) &\text{for } x \not\in \mathcal{D}.
                \end{cases}
            \end{align*}

        The time dependent contrast parameter associated to the resonator $D_i^m$ will be denoted by $\delta_i^m(t):=\rho_i^m(t)/\rho_b(t)$ and which we will assume satisfies for some general parameter $\delta$
            \begin{equation}
             \delta_i^m = O(\delta) \text{ as } \delta  \rightarrow 0,\end{equation}
        uniformly in $t \in \mathbb{R},~ m \in \mathbb{Z},~ i =1,2$.
        In the setting of  instantly changing high-contrast materials, the following definition will be useful.
        
        \begin{definition}[Subwavelength quasi-harmonic solution]
            We will call a quasi-harmonic solution $u(x,t)= v(x)\exp(-i \omega(t)t)$ of the time dependent wave equation \eqref{eq:wave_eq} a \emph{subwavelength quasi-harmonic solution} if $(v(x),\omega^-)$ and $(v(x),\omega^+)$ are subwavelength solutions to the associated Helmholtz problems. 
            If furthermore $v(x)$ is a localized or defect mode, we call $\omega(t)$ the \emph{time-dependent defect mode eigenfrequency} of $v(x)$.
        \end{definition}

        \begin{lemma}\label{lem:time-dep_subwavelength}
            Let $\kappa(x,t)$ and $\rho(x,t)$ be defined as above and consider the wave equation \eqref{eq:wave_eq}. If the wave equation \eqref{eq:wave_eq} admits subwavelength solutions for negative times and if
                $$b~\kappa(x,t)\vert_{t<0} = \kappa(x,t)\vert_{t\geq0},$$
                $$\frac{1}{b}~\rho(x,t)\vert_{t<0} = \rho(x,t)\vert_{t\geq0},$$
            for some constant $b \in \mathbb{C}^*$ which is independent of $\delta$, then the wave equation \eqref{eq:wave_eq} admits quasi-harmonic subwavelength solutions.
        \end{lemma}

        \begin{proof}
            The proof is a direct consequence of Theorem \ref{thm:time_dependent}. Let $(v(x),\omega^-)$ be a subwavelength solution to the Helmholtz problem for negative times. Then $$ u(x,t) = \begin{cases}
                v(x)\exp(-i\omega^- t) &\text{for } t<0,\\
                v(x)\exp(-i\omega^+ t) &\text{for } t\geq0,\\
            \end{cases}$$ is a quasi-harmonic subwavelength solution of \eqref{eq:wave_eq} if and only if $\omega^+ = b\omega^-$.
        \end{proof}
        
        Having clarified the behavior of subwavelength quasi-harmonic solutions to the wave equation of an instantly changing material, we will present a way to design metamaterials which admit quasi-harmonic subwavelength solutions of prescribed time-dependent defect mode eigenfrequency. This result then lays the basis for the construction of instantly changing materials which admit spatio-temporally localized quasi-harmonic subwavelength solutions.

        \begin{theorem}\label{thm:time_dep_material_design}
            Let $\mathcal{D}$ and $D_i^m$ be defined as at the beginning of this section and use the same notation as in Theorem \ref{thm:formula_defect_freq} which will describe the material for $t<0$. Then there exists a map $\Phi$ which associates to a pair of frequencies $(\omega^-,\omega^+)$ the defect material parameter $V^{def}_1$ and the constant $b$ from Lemma \ref{lem:time-dep_subwavelength}, such that $\omega(t) = \chi_{t<0}\omega^-+\chi_{t\geq0}\omega^+$ is the time-dependent defect mode eigenfrequency of a quasi-harmonic subwavelength solution $u(x,t)= v(x)\exp(-i \omega(t)t)$ of the associated instantly changing material. More precisely, let $V_1$, $V_2$ be as in Theorem \ref{thm:formula_defect_freq}
                \begin{align*}
                    \begin{matrix}
                        \Phi: &\mathbb{C}\setminus \left(\bigcup_{\alpha\in Y^*}\sigma(\mathcal{C}^\alpha)\cup W\right) \times \mathbb{C}^* &\longrightarrow &\mathbb{C}\times\mathbb{C}\\
                        \\
                        &(\omega^-,\omega^+) &\longmapsto &(V^{def}_1(\omega^-),b),
                    \end{matrix}
                \end{align*}
            where $V^{def}_1(\omega^-) = \Psi(\omega^-)$ as in Theorem \ref{thm:formula_defect_freq} and $b = \omega^+/\omega^-$. If we let $V_1,V_2,V_1^{def}$ be the material parameters for $t<0$ and  $b^2 V_1,b^2 V_2,b^2 V_1^{def}$ be the material parameters for $t \geq 0$, then this system admits a quasi-harmonic subwavelength solution $u(x,t)$ of the form 
                \begin{align*}
                    u(x,t) = v(x)\exp(-i \omega(t)t) \text{ with } \omega(t) = \begin{cases}
                        \omega^- &\text{if } t < 0,\\
                        \omega^+ &\text{if } t \geq 0.
                    \end{cases}
                \end{align*}
        \end{theorem}

        \begin{proof}
            The result is a combination of Theorems \ref{thm:formula_defect_freq} and \ref{thm:time_dependent}. Using Theorem \ref{thm:formula_defect_freq}, one obtains the defect parameter $V^{def}_1(\omega^-)$ such that $\omega^-$ occurs as a defect mode eigenfrequency. In order to construct a time-dependent structure which admits $u(x,t)$ as a solution, it follows from Theorem \ref{thm:time_dependent} that the only possible choice of compressibility and density for positive times is given by $\kappa^+ = b\kappa^-,\rho^+=\rho^-/b$, respectively. Thus, if the material parameters for positive times satisfy
            $$ (\delta^m_i)^+((v^m_i)^+)^2 = \frac{(\kappa_m^i)^+}{(\rho_b)^+} = b^2 \frac{(\kappa_m^i)^-}{(\rho_b)^-} = b^2 (\delta_m^i)^-((v^m_i)^-)^2 \quad \text{ for } m \in \mathbb{Z}, i = 1,2,$$
            that is, if the material parameters for positive times are given by $b^2 V_1,b^2 V_2,b^2 V_1^{def}$, then $\omega(t)$ is a time-dependent defect mode eigenfrequency of the associated instantly changing material.
        \end{proof}

        \subsection{Spatio-temporal localization}\label{subsec:Spatio-temporal_localization}

        In this section, the results from Theorem \ref{thm:time_dep_material_design} will be used to design (and thus show existence of) materials which admit spatio-temporally localized modes.
        The following definition will clarify what we understand by a spatio-temporally localized mode.

        \begin{definition}[Spatio-temporal localization]
            We  call a quasi-harmonic solution $u(x,t)=v(x)\exp(-i\omega(t)t)$ to the wave equation \eqref{eq:wave_eq} \emph{spatio-temporally localized}, if $v(x)$ is a localized mode and if for almost every $x \in \mathbb{R}^3$ the function $t \mapsto u(x,t)$ is square integrable.
        \end{definition}

        The following characterization will help for the design of instantly changing materials which admit spatio-temporally localized modes.

        \begin{prop}\label{prop:spatio-temp_loc}
            Let $u(x,t)=v(x)\exp(-i\omega(t)t)$ be a quasi-harmonic solution to the wave equation \eqref{eq:wave_eq} of an instantly changing material given by \eqref{eq:time_dep_rho_kappa}.
            Then, $u(x,t)$ is a spatio-temporally localized if and only if $v(x)$ is a localized mode of the associated Helmholtz problem for negative and positive times and if $\Im(\omega^-)>0$ and $\Im(\omega^+)<0$.
        \end{prop}

        \begin{proof}
            Let $u(x,t)=v(x)\exp(-i\omega(t)t)$ be spatio-temporally localized. Then, $v(x)$ is a localized mode and for almost every $x \in \mathbb{R}^3$ the function $t \mapsto v(x)\exp(-i\omega(t)t)$ is square integrable. This is precisely the case when $t \mapsto \exp(-i\omega(t)t)$ is square integrable, which occurs precisely when $\Im(\omega^-)>0$ and $\Im(\omega^+)<0$. Conversely, if $v(x)$ is a localized mode and $\Im(\omega^-)>0$ and $\Im(\omega^+)<0$, then $u(x,t)$ is spatio-temporally localized.
        \end{proof}

        This proposition allows us to finally state the following design and existence result for spatio-temporal localization.
        
        \begin{theorem}[Design for spatio-temporal localization]\label{thm:spatio-temporal_loc}
        Assuming the same setting and notation as in Theorem \ref{thm:time_dep_material_design}.
        Denote by $\mathcal{H}\subset \mathbb{C}$ the upper half-plane of the complex numbers given by $\mathcal{H}:= \{x \in \mathbb{C}:\Im(x)>0\}$ and denote by $\overline{\mathcal{H}}$ the lower half plane. 
        Let $$(\omega^-,\omega^+) \in \mathcal{H}\setminus \left(\bigcup_{\alpha\in Y^*}\sigma(\mathcal{C}^\alpha)\cup W\right) \times \overline{\mathcal{H}}.$$
        Then, there exists an instantly changing material that admits $\omega(t) = \chi_{t<0}\omega^- + \chi_{t\geq0}\omega^+$ as a time-dependent defect mode eigenfrequency associated to a spatio-temporally localized quasi-harmonic subwavelength solution $u(x,t)$.
        
        More precisely, let $V_1^{def}(\omega^-) := \Phi_1(\omega^-,\omega^+)$ and $b = \Phi_2(\omega^-,\omega^+)$ and let $V_1$, $V_2$ be as in Theorem \ref{thm:formula_defect_freq}. If we let $V_1,V_2,V_1^{def}$ to be the material parameters for $t<0$ and  $b^2 V_1,b^2 V_2,b^2 V_1^{def}$ to be the material parameters for $t \geq 0$, then this system admits a spatio-temporally localized quasi-harmonic subwavelength solution $u(x,t)$ of the form 
                \begin{align*}
                    u(x,t) = v(x)\exp(-i \omega(t)t) \text{ with } \omega(t) = \begin{cases}
                        \omega^- &\text{if } t < 0,\\
                        \omega^+ &\text{if } t \geq 0.
                    \end{cases}
                \end{align*}
        \end{theorem}

        \begin{proof}     
            This is a straightforward combination of the results of Proposition \ref{prop:spatio-temp_loc} and Theorem \ref{thm:time_dep_material_design}.
        \end{proof}
        
\section{Concluding remarks}
    
    In this paper, we have derived closed formulas for the design of defected metamaterials which admit specified  defect mode eigenfrequencies. This has been achieved both in the setting of static metamaterials and in the setting of instantly changing metamaterials.
Our results do not rely on the fact that a \emph{chain} of resonator dimers is considered. That is, in Lemmas \ref{lem:formula_defect_freq_single} and \ref{lem:formula_defect_freq_double}, the sublattice $\Lambda = (1,0,0)^T\mathbb{Z}$ can be replaced with any $1,2$ or $3$-dimensional sublattice of $\mathbb{R}^3$ without changing the results. The characterizing equations from Lemmas \ref{lem:formula_defect_freq_single} and \ref{lem:formula_defect_freq_double} are the basis for the derivation of \emph{exact formulas} for the design of a defected structure that admits a given frequency $\omega$ or a given frequency pair $(\omega_1,\omega_2)$ as defect mode eigenfrequencies (see Theorems \ref{thm:formula_defect_freq} and \ref{thm:formula_double_defect_freq}). It is important to note that the same procedures as for Theorems \ref{thm:formula_defect_freq} and \ref{thm:formula_double_defect_freq} are possible for three and four prescribed frequencies and three and four defected resonators, respectively. For more than four frequencies the system of equations \eqref{eq:material_design_multiple_defect} cannot be generally solved as in Theorems \ref{thm:formula_defect_freq} and \ref{thm:formula_double_defect_freq}, but has to be solved numerically in order to obtain the defect material parameters.

    Having developed a method to have exact formulas for the metamaterial design of defected materials with prescribed defect mode eigenfrequencies, we have used our findings to obtain in Section \ref{sec:Spatio-temporal_localization} similar results in the setting of instantly changing materials with prescribed time-dependent defect mode eigenfrequency $\omega(t) = \omega^{-}\chi_{t<0}+\omega^{+}\chi_{t\geq0}$. To this end, in Section \ref{subsec:Instantly_changing_material}, we have characterized  in Theorem \ref{thm:time_dependent} the occurrence of quasi-harmonic waves in instantly changing materials. After having established the material design of static materials with prescribed defect mode eigenfrequencies in Theorem \ref{thm:formula_defect_freq} and characterized the occurrence of quasi-harmonic solutions in time-dependent materials in Theorem \ref{thm:time_dependent}, we have presented the design of instantly changing materials which admit localized modes with a prescribed time-dependent defect mode eigenfrequency $\omega(t) = \omega^{-}\chi_{t<0}+\omega^{+}\chi_{t\geq0}$. The result is presented in Theorem \ref{thm:time_dep_material_design}. The design in Section \ref{subsec:Spatio-temporal_localization} of instantly changing materials  which admit spatio-temporally localized modes results as an application of this Theorem. The main finding is presented  in Theorem \ref{thm:spatio-temporal_loc}.

\section*{Acknowledgements}
We would like to thank Bryn Davies for kindly providing his capacitance matrix code for structures with finitely many resonators.

\bibliographystyle{abbrv}
\bibliography{references}{}

\end{document}